\documentclass[final,leqno]{siamltex704}
\usepackage{amsmath}
\usepackage{graphicx}
\usepackage[notcite,notref]{showkeys}
\usepackage{mathrsfs}
\usepackage{color}
\usepackage{float}
\usepackage{amsfonts,amssymb}
\usepackage{dsfont}
\usepackage{pifont}
\usepackage{wrapfig} 
\usepackage{hyperref}
\usepackage{multirow}
\numberwithin{equation}{section}
\def\3bar{{|\hspace{-.02in}|\hspace{-.02in}|}}
\def\E{{\mathcal{E}}}
\def\T{{\mathcal{T}}}

\def\pT{{\partial T}}

\def\W{{\mathcal{W}}}

\def\bn{{\mathbf{n}}}

\def\ljump{{[\![}}
\def\rjump{{]\!]}}

\def\bbeta{{\boldsymbol{\beta}}}

\newtheorem{algorithm}{Primal-Dual Weak Galerkin Algorithm}[section]

\title {A Primal-dual weak Galerkin finite element method for linear convection equations in non-divergence form}

\begin{document}

\author{
Dan Li \thanks{Department of Applied Mathematics, Northwestern Polytechnical University, Xi'an, Shannxi 710072, China.}
\and
Chunmei Wang \thanks{Department of Mathematics \& Statistics, Texas Tech University, Lubbock, TX 79409, USA (chunmei.wang@ttu.edu). The research of Chunmei Wang was partially supported by National Science Foundation Award DMS-1849483.}
\and
Junping Wang\thanks{Division of Mathematical
Sciences, National Science Foundation, Alexandria, VA 22314
(jwang@nsf.gov). The research of Junping Wang was supported in part by the
NSF IR/D program, while working at National Science Foundation.
However, any opinion, finding, and conclusions or recommendations
expressed in this material are those of the author and do not
necessarily reflect the views of the National Science Foundation.}}

\maketitle

\begin{abstract}
This article presents a new primal-dual weak Galerkin (PD-WG) finite element method for first-order linear convection equations in non-divergence form. The PD-WG method is based on the locally reconstructed differential operators (e.g., discrete weak gradients) developed in the weak Galerkin context for both the primal differential operator and its adjoint, with stabilizers employed to enhance the stability of the numerical scheme. The numerical method results in a symmetric discrete linear system involving not only the primal variable, but also the dual variable (also known as the Lagrangian multiplier) for the adjoint equation. Optimal order of error estimates are derived in various discrete Sobolev norms for the numerical solutions arising from the PD-WG scheme. Numerical results are reported to illustrate the accuracy and efficiency of the new PD-WG method.
\end{abstract}



\pagestyle{myheadings}

\section{Introduction}
This paper is concerned with a new development of numerical methods for first-order linear convection equations in non-divergence form by using discontinuous finite element functions. For simplicity, consider the model problem of seeking an unknown function $\lambda$ satisfying
\begin{equation}\label{model}
\begin{split}
 \bbeta(x)\cdot\nabla\lambda-c(x)\lambda=&f \quad \text{in}\ \Omega,\\
\lambda=&g\quad \text{on}\  \Gamma_-,
\end{split}
\end{equation}
where $\Omega$ is an open bounded and connected domain in $\mathbb R^d \ (d=2, 3)$ with Lipschitz continuous boundary $\Gamma=\partial \Omega$, $\Gamma_-$ is the inflow portion of the domain boundary given by the condition of $\bbeta \cdot \bn<0$ with $\bn$ being the unit outward normal direction to $\Gamma$. Assume that the convection vector $\bbeta=(\beta_1, \cdots, \beta_d) \in [L^{\infty}(\Omega)]^d$, the reaction coefficient $c\in L^{\infty}(\Omega)$, the load function $f\in L^2(\Omega)$, and the inflow boundary data $g\in L^2(\Gamma_-)$.

The first-order linear partial differential equations (PDEs) of hyperbolic-type are also known as transport equations or linear convection equations which arise in many areas of science and engineering. Numerical methods for linear convection equations often face a grand challenge on its stability and capability of resolving the solution's  discontinuity and the sharp changing front in scientific computing. The linear convection equation also serves as an excellent benchmark for testing new ideas in numerical PDEs. Readers are referred to the introduction section of \cite{wwhyperbolic} and the references cited therein for a detailed description of the first-order linear convection equation and some of its physical applications.

This paper aims to develop a new numerical method for the linear convection problem \eqref{model} for which the convection vector $\bbeta$ and the reaction coefficient $c$ are assumed to be piecewise smooth functions without any additional coercivity assumption in the form of $c+\frac12 \nabla\cdot\bbeta \ge \alpha>0$ or alike as commonly seen in most existing literatures in the numerical study. The new numerical scheme is devised by following the framework of the primal-dual weak Galerkin (PD-WG) finite element method introduced and studied in \cite{ww2016, ww2017, ww2018, wz2019, wwhyperbolic, w2018}. The PD-WG method was originally formulated as a constraint optimization problem where the optimization involves minimizing the ``discontinuity" of the approximating solution with constraints given by a local satisfaction of the underlying PDEs on each element. The resulting Euler-Lagrange formulation gives rise to a symmetric system involving both the primal (original) equation and the dual (adjoint) equation interconnected by carefully-chosen stabilizers that provide ``weak continuity" or ``weak smoothness" necessary for convergence and stability. It should be noted that the approach of PD-WG finite element method for solving PDEs was also developed by Burman \cite{Burman2013, burman2014} in other finite element contexts, and it was given the name of {\em ``stabilized finite element methods"} by Burman. Recent studies have shown that the PD-WG finite element methods have great potentials for PDEs for which the traditional variational formulations are either not available or difficult to use for numerical discretization.

Let us briefly discuss the main idea behind the PD-WG method for the first-order linear convection problem \eqref{model}. A weak formulation for the model problem (\ref{model}) can be given by seeking $\lambda \in H_{\bbeta}^1(\Omega)$ such that $\lambda=g$ on $\Gamma_-$ and
\begin{equation}\label{weakform}
(\bbeta\cdot\nabla\lambda-c\lambda, v)=(f, v), \qquad \forall v\in L^2(\Omega),
\end{equation}
where $H_{\bbeta}^1(\Omega)=\{w \in L^2(\Omega): \bbeta\cdot\nabla w \in L^2(\Omega)\}$. The weak gradient operator $\nabla_w$ (see \cite{wy3655} for definition) can be used to reformulate \eqref{weakform} as follows:
\begin{equation}\label{weakform-02}
(\bbeta\cdot\nabla_w\{\lambda\}-c\lambda, v)=(f, v), \qquad \forall v\in L^2(\Omega),
\end{equation}
where $\{\lambda\} = \{\lambda|_T, \lambda|_\pT\}$ is defined as a weak function in the weak Galerkin context. The weak function has approximations by piecewise polynomials on each element $T$ as well as on its boundary $\pT$. The weak gradient $\nabla_w \{\lambda\}$ can be discretized by using vector-valued polynomials, denoted as $\nabla_{w, h} \{\lambda\}$. The variational problem (\ref{weakform-02}) can then be approximated by seeking $\lambda_h =\{\lambda_0,\lambda_b\}\in W_h$ such that $\lambda_b=Q_bg$ on $\Gamma_-$ and satisfying
\begin{equation}\label{EQ:10-12-2015:01}
(\bbeta\cdot\nabla_{w, h} \lambda_h-c\lambda_0, v)=(f, v), \qquad \forall v\in M_h,
\end{equation}
where $W_h$ is a trial space for the weak functions and $M_h$ is a test space consisting of piecewise polynomials. The problem (\ref{EQ:10-12-2015:01}) is well-posed if the {\em inf-sup} condition of Babu\u{s}ka \cite{babuska} is satisfied. For most readily available trial and test spaces, the {\em inf-sup} condition of Babu\u{s}ka is not satisfied so that the problem (\ref{EQ:10-12-2015:01}) is not well posed as it stands. A remedy of this trouble is to consider the dual equation of \eqref{EQ:10-12-2015:01} that seeks $u_h\in M_h$ satisfying
\begin{equation}\label{EQ:09-12-2018:01}
(\bbeta\cdot\nabla_{w, h} \sigma -c\sigma_0, u_h)=0, \qquad
\forall \sigma \in W_h^{0},
\end{equation}
where $W_h^{0}$ is the subspace of $W_h$ consisting of all weak functions with vanishing boundary value on $\Gamma_-$. A formal coupling of \eqref{EQ:10-12-2015:01} and \eqref{EQ:09-12-2018:01} results in the following primal-dual weak Galerkin scheme: Find $\lambda_h\in W_h$ and $u_h\in M_h$, such that $\lambda_b=Q_bg$ on $\Gamma_-$ and
\begin{equation}\label{primal-dual-wg}
\left\{
\begin{split}
s(\lambda_h, \sigma)+(\bbeta\cdot\nabla_{w, h} \sigma -c\sigma_0, u_h)=\sum_{T\in {\cal T}_h} \tau_1 (f, \bbeta \cdot \nabla \sigma_0-c\sigma_0), \qquad \forall \sigma \in W_h^{0},\\
-\tau_2 \sum_{T\in \T_h} h_T^2(u_h, v)_T+(\bbeta\cdot\nabla_{w, h} \lambda_h-c\lambda_0, v)=(f, v), \qquad \forall v\in M_h,
 \end{split}\right.
\end{equation}
where $s(\cdot, \cdot)$ is a stabilizer/smoother that enforces necessary weak continuity for the numerical solution $\lambda_h$. In fact, the stabilizer aims to measure the level of ``continuity" of the weak function $\sigma \in W_{h}$ in the sense that $\sigma \in W_{h}$ is a classical $C^0$-conforming element if and only if $s(\sigma, \sigma)=0$.

The linear convection equation in non-divergence form \eqref{model} is essentially the adjoint of the model problem in divergence form discussed in \cite{wwhyperbolic}. In the development of the PD-WG method in \cite{wwhyperbolic}, the linear convection equation in divergence form was characterized by a weak form obtained through the usual integration by parts so that no partial derivatives are taken on the primal variable. The corresponding numerical scheme is therefore of low-regularity in nature, as convergence occurs under any $H^\theta$-regularity for the exact solution with $\theta >0$. For the linear convection equation in non-divergence form \eqref{model}, we shall and have to use a straightforward weak form through a simple test against any square integrable function. Due to the difference of characterization for the primal equation, a least-squares term in the form of $\tau_2 \sum_{T\in \T_h} h_T^2(u_h, v)$ is added to the primal equation in the new PD-WG scheme \eqref{primal-dual-wg} in order to achieve an optimal order of convergence in $L^2$. Furthermore, a term like $\sum_{T\in {\cal T}_h} \tau_1 (f, \bbeta \cdot \nabla \sigma_0-c\sigma_0)$ is added to the right-hand side of the dual equation due to the corresponding least squares term in the stabilizer $s(\cdot, \cdot)$ to be detailed in later sections. Optimal order of error estimates are established in various discrete Sobolev norms for the new PD-WG finite element method.

The main contributions of this paper are the following. First of all, a new numerical method was designed and analyzed thoroughly for its stability and convergence. Secondly, the mathematical theory was established based on a minimal assumption on the linear convection equation and its coefficients; namely, the original problem has one and only one solution and the coefficients are merely piecewise smooth. Due to the global non-smoothness of the convection vector $\bbeta$, the linear convection equation in non-divergence form can not be formulated in a divergence form for any possible application of the numerical method developed in  \cite{wwhyperbolic}. Therefore, the result of this paper fills a gap of numerical study of PDEs in the realm of weak Galerkin finite element methods.

The paper is organized as follows. In Section 2, we present the primal-dual weak Galerkin algorithm for solving the first-order linear convection problem in non-divergence form. In Section 3, we shall prove the existence and uniqueness for the numerical solutions. In Section 4, we shall derive an error equation for the PD-WG finite element solutions. In Sections 5 and 6, we shall derive some optimal order error estimates for the PD-WG approximations in some discrete Sobolev norms. In Section 7, we report some numerical results to demonstrate the efficiency, stability, and accuracy of the new PD-WG method.

Throughout the paper, we follow the usual notation for Sobolev spaces and norms. For any open bounded domain $D\subset \mathbb{R}^d$ ($d$-dimensional Euclidean space) with Lipschitz continuous boundary, we use $\|\cdot\|_{s,D}$ and $|\cdot|_{s,D}$ to denote the norm and seminorm in the Sobolev space $H^s(D)$ for any $s\ge 0$, respectively. The inner product in
$H^s(D)$ is denoted by $(\cdot,\cdot)_{s, D}$. The space $H^0(D)$ coincides with $L^2(D)$, for which the norm and the inner product are denoted by $\|\cdot \|_{D}$ and $(\cdot,\cdot)_{D}$, respectively. When $D=\Omega$, or when the domain of integration is clear from the context, the subscript $D$ is dropped in the norm and inner product notations.

\section{Primal-Dual Weak Galerkin Algorithm}
\subsection{Discrete weak gradient}
Let $T$ be a polygonal or polyhedral domain with boundary $\partial T$. A weak function on $T$ is defined as a pair $v=\{v_0, v_b \}$ such that $v_0\in L^2(T)$ and $v_b\in L^{2}(\partial T)$. The first component $v_0$ can be viewed as the value of $v$ in the interior of $T$, while the second component $v_b$ represents $v$ on the boundary of $T$. In general, $v_b$ may not necessarily be the trace of $v_0$ on $\partial T$, though being the trace of $v_0$ on $\partial T$ would be viable option for $v_b$.
Denote by $\W(T)$ the local space of all weak functions on $T$; i.e.,
\begin{equation*}\label{2.1}
\W(T)=\{v=\{v_0, v_b\}: v_0 \in L^2(T), v_b \in L^{2}(\partial T)\}.
\end{equation*}

The weak gradient of $v\in \W(T)$, denoted by $\nabla_w v$, is defined as a linear functional in $[H^1(T)]^d$ such that
\begin{equation*}
(\nabla_w  v,\boldsymbol{\psi})_T :=-(v_0,\nabla \cdot \boldsymbol{\psi})_T+\langle v_b,\boldsymbol{\psi}\cdot \textbf{n}\rangle_{\partial T},\qquad \forall \boldsymbol{\psi}\in [H^1(T)]^d.
\end{equation*}
Denote by $P_r(T)$ the space of polynomials on $T$ with degree
$r$ and less. A discrete version of $\nabla_{w} v$  for $v\in \W(T)$, denoted as $\nabla_{w, r, T}v$, is defined as the unique polynomial vector in $[P_r(T) ]^d$ satisfying
\begin{equation}\label{disgradient}
(\nabla_{w, r, T} v, \boldsymbol{\psi})_T=-(v_0, \nabla \cdot \boldsymbol{\psi})_T+\langle v_b, \boldsymbol{\psi} \cdot \textbf{n}\rangle_{\partial T}, \quad\forall\boldsymbol{\psi}\in [P_r(T)]^d.
\end{equation}
From the integration by parts, we may rewrite (\ref{disgradient}) as follows
\begin{equation}\label{disgradient*}
(\nabla_{w, r, T} v, \boldsymbol{\psi})_T= (\nabla v_0, \boldsymbol{\psi})_T-\langle v_0- v_b, \boldsymbol{\psi} \cdot \textbf{n}\rangle_{\partial T}, \quad\forall\boldsymbol{\psi}\in [P_r(T)]^d,
\end{equation}
provided that $v_0\in H^1(T)$.

\subsection{Numerical algorithm}\label{Section:WGFEM}
Let ${\cal T}_h$ be a partition of the domain $\Omega$ into polygons in 2D or polyhedra in 3D which is shape regular in the sense of \cite{wy3655}. Denote by ${\mathcal E}_h$ the set of all edges/faces in ${\cal T}_h$, and ${\mathcal E}_h^0={\mathcal E}_h \setminus \partial\Omega$ be the set of all interior edges/faces. Denote by $h_T$ the size of $T\in {\cal T}_h$ and
$h=\max_{T\in {\cal T}_h}h_T$ the meshsize of the partition ${\cal T}_h$. For any piecewise smooth function $\phi$  with respect to the partition $\T_h$, denote by $\ljump \phi\rjump$ the jump of $\phi$ along the interior edge/face $e$ given by
$$
\ljump \phi\rjump =\phi_1 \bn_1+\phi_2\bn_2,
$$
where $\phi_i:=\phi|_{T_i}$, and $\bn_i$ is the unit outward normal direction on $e=\partial T_1\cap \partial T_2$ pointing exterior to the element $T_i,\ i=1, 2$.

For any given integer $k\geq 1$, denote by $W_k(T)$ the local space of discrete weak functions; i.e.,
$$
W_k(T)=\{\{\sigma_0,\sigma_b\}:\sigma_0\in P_k(T),\sigma_b\in P_k(e), e\subset \partial T\}.
$$
Patching $W_k(T)$ over all the elements $T\in {\cal T}_h$ through a common value $v_b$ on the interior interface $\E_h^0$, we arrive at a global weak finite element space $W_h$. Denote by $W_h^{0}$ the subspace of $W_h$ with vanishing boundary values on $\Gamma_-$; i.e.,
$$
W_h^{0}=\{\{\sigma_0, \sigma_b\}\in W_h: \sigma_b|_{e}=0, e\subset \Gamma_-\}.
$$
Next, let $M_h$ be the finite element space of piecewise polynomials of degree $k-1$; i.e.,
$$
M_h=\{w: w|_T\in P_{k-1}(T), \forall T\in {\cal T}_h\}.
$$

For simplicity of notation and without confusion, for any $\sigma\in W_h$, denote by $\nabla_{w}\sigma$ the discrete weak gradient
$\nabla_{w, k-1, T}\sigma$ computed  by
(\ref{disgradient}) on each element $T$; i.e.,
$$
(\nabla_{w}\sigma)|_T= \nabla_{w, k-1, T}(\sigma|_T), \qquad \forall T\in \T_h.
$$
In the space $W_h$ and $M_h$, we introduce the following bilinear forms
\begin{equation}\label{stab1}
\begin{split}
s(\rho, \sigma)=&\sum_{T\in {\cal T}_h}\int_{\partial T}h_T^{-1} (\rho_0-\rho_b)(\sigma_0-\sigma_b)ds\\
&+\tau_1 \int_{ T} (\bbeta\cdot \nabla \rho_0-c \rho_0) (\bbeta\cdot\nabla \sigma_0-c\sigma_0)dT,
\end{split}
\end{equation}
\begin{equation*}
b(\sigma, v)= \sum_{T\in {\cal T}_h}(\bbeta \cdot \nabla_w \sigma-c\sigma_0, v)_T,
\end{equation*}
where $\rho, \ \sigma\in W_h$,  $v\in M_h$, and $\tau_1 \geq 0$ is a parameter.

We are now in a position to state the numerical scheme for the first-order linear convection problem (\ref{model}) in the framework of primal-dual weak Galerkin as follows:
\begin{algorithm}\label{PDWG1}
Find $(\lambda_h, u_h) \in W_h \times M_{h}$, such that $\lambda_b|_e = Q_b (g|_e)$ for all edge/face $e\subset \Gamma_-$ and satisfying
\begin{eqnarray}\label{al-general}
 &&s(\lambda_h, \sigma)+b(\sigma, u_h)=\sum_{T\in {T_h}}\tau_1 (f, \bbeta\cdot \nabla \sigma_0-c\sigma_0)_T, \quad  \forall\sigma\in W_{h}^{0}, \\
&&-\tau_2 \sum_{T\in {\cal T}_h} h_T^2(u_h, v)_T+ b(\lambda_h, v)=(f, v),\qquad \forall v\in M_h,\label{al-general-2}
\end{eqnarray}
where $\tau_2>0$ is a parameter and $Q_b$ is the local $L^2$ projection operator into $P_k(e)$.
\end{algorithm}

\section{Solution Existence and Uniqueness}\label{Section:EU}
On each element $T\in\T_h$, let $Q_0$ be the $L^2$ projection onto $P_k(T)$, where $k\geq 1$ is any given integer. On each edge/face $e\subset\partial T$, let $Q_b$ be the $L^2$ projection operator onto $P_{k}(e)$. For any $w\in H^1(\Omega)$, denote by $Q_h w$ the $L^2$ projection into the weak finite element space $W_h$ such that
$$
(Q_hw)|_T:=\{Q_0(w|_T),Q_b(w|_{\pT})\},\qquad \forall T\in\T_h.
$$
Denote by ${\cal Q}_h$ the $L^2$ projection operator onto the finite element space $M_h$.

\begin{lemma}\label{Lemma5.1} \cite{wy3655} The $L^2$ projection operators $Q_h$ and ${\cal Q}_h$ satisfy the following commutative property:
\begin{equation}\label{l}
\nabla_{w}(Q_h w) = {\cal Q} _h(\nabla w), \qquad \forall w\in H^1(T).
\end{equation}
\end{lemma}

For simplicity of analysis, assume that the convection vector $\bbeta$ and the reaction coefficient $c$ are piecewise smooth functions with respect to the finite element partition $\T_h$ in the rest of the paper.

\begin{theorem}\label{thmunique1} Assume that the first-order linear convection problem \eqref{model} has a unique solution. The primal-dual weak Galerkin algorithm (\ref{al-general})-(\ref{al-general-2}) has a unique solution for any parameter $\tau_1>0$.
\end{theorem}
\begin{proof} It suffices to show that the homogeneous problem of (\ref{al-general})-(\ref{al-general-2}) has only the trivial solution. To this end, assume $f=0$ and $g=0$. By choosing $v=u_h$ and $\sigma=\lambda_h$ in (\ref{al-general})-(\ref{al-general-2}) we arrive at
$$
s(\lambda_h, \lambda_h)+\tau_2\sum_{T\in {\cal T}_h} h_T^2(u_h, u_h)_T=0,
$$
which implies $\lambda_0=\lambda_b$ on each $\partial T$, $\bbeta \cdot \nabla \lambda_0-c\lambda_0=0$, and $u_h=0$ on each element $T$. We thus obtain $\lambda_0\in C^0(\Omega )$ and furthermore $\bbeta \cdot  \nabla \lambda_0-c\lambda_0=0$ in $\Omega$, which, together with $\lambda_0=\lambda_b=0$ on $\Gamma_-$, yields $\lambda_0\equiv 0$ in $\Omega$ from the solution uniqueness of the model problem (\ref{model}). From $\lambda_0=\lambda_b$ on each  $\partial T$, we have $\lambda_b\equiv0$ in $\Omega$ so that $\lambda_h\equiv 0$ in $\Omega$. Additionally, it follows from $u_h\in M_h$ and  $u_h=0$ on each element $T$ that $u_h \equiv 0$ in $\Omega$.
This completes the proof of the theorem.
\end{proof}

The first-order linear convection problem \eqref{model} with homogeneous boundary condition (i.e., $g=0$) is said to satisfy a local $H^1$-regularity assumption if there exists a non-overlapping partition of the domain $\Omega = \bigcup_{i=1}^J \Omega_i$ such that the solution $\lambda$ exists, $\lambda|_{\Omega_i}\in H^1(\Omega_i)$ for $i=1, \cdots, J$, and
\begin{equation}\label{regula}
\Big(\sum_{i=1}^J\|\lambda\|^2_{1, \Omega_i}\Big)^{\frac{1}{2}}\leq C\|f\|_0,
\end{equation}
where $C$ is a generic constant.

\begin{theorem}\label{thmunique2}
Assume that
The primal-dual weak Galerkin algorithm  (\ref{al-general})-(\ref{al-general-2}) has a unique solution when the parameter $\tau_1=0$ is taken, provided that the meshsize $h$ is sufficiently small such that $h<h_0$ for a fixed but sufficiently small $h_0$.
\end{theorem}
\begin{proof} It suffices to show that the homogeneous problem of (\ref{al-general})-(\ref{al-general-2}) has only the trivial solution. To this end, assume $f=0$ and $g=0$. As $\tau_1=0$, by choosing $v=u_h$ and $\sigma=\lambda_h$ in (\ref{al-general})-(\ref{al-general-2}) we arrive at $$
s(\lambda_h, \lambda_h)+\tau_2\sum_{T\in {\cal T}_h} h_T^2(u_h, u_h)_T=0,
$$
which leads to $\lambda_0=\lambda_b$ on each $\partial T$ and $u_h=0$ on each element $T\in {\cal T}_h$. It follows from (\ref{al-general-2}) that
\begin{equation*}
\begin{split}
0=&b(\lambda_h, v)\\
=& \sum_{T\in {\cal T}_h}  (\bbeta \cdot \nabla_w \lambda_h-c\lambda_0, v)_T\\
=& \sum_{T\in {\cal T}_h}  (\bbeta \cdot \nabla  \lambda_0-c\lambda_0, v)_T\\
=& \sum_{T\in {\cal T}_h}  ({\cal Q}_h(\bbeta \cdot \nabla  \lambda_0-c\lambda_0), v)_T\\
\end{split}
\end{equation*}
for all $v\in M_h$, where we have used $\nabla_w \lambda_h=\nabla  \lambda_0$ due to the fact that $\lambda_0=\lambda_b$ on each $\partial T$ plus the identity \eqref{disgradient*}.  Therefore, we obtain ${\cal Q}_h(\bbeta \cdot \nabla  \lambda_0-c\lambda_0)=0$ on each $T\in {\cal T}_h$ by letting $v={\cal Q}_h(\bbeta \cdot \nabla  \lambda_0-c\lambda_0)\in M_h$. From $\lambda_0=\lambda_b$ on each  $\partial T$ we have $\lambda\in C^0(\Omega)$. Thus,
\begin{equation}\label{au1}
\bbeta \cdot \nabla  \lambda_0-c\lambda_0=(I-{\cal Q}_h)(\bbeta \cdot \nabla  \lambda_0-c\lambda_0):=F, \quad \lambda_0|_{\Gamma_-}=0.
\end{equation}

Note that the convection coefficient $\bbeta$ is piecewise smooth with respect to the partition $\T_h$. Thus, from the regularity assumption (\ref{regula}), the estimate (\ref{error2}) with $m=1$ and $m=0$, the inverse inequality, and the equation (\ref{au1}), we obtain
\begin{equation*}
\begin{split}
 \left(\sum_{i=1}^J\|\lambda_0\|^2_{1, \Omega_i}\right)^{\frac{1}{2}} \leq  & C\|F\|_0\\
\leq &C\|(I-{\cal Q}_h)(\bbeta \cdot \nabla  \lambda_0-c\lambda_0)\|\\
\leq &C \|(I-{\cal Q}_h)((\bbeta-\overline{\bbeta}) \cdot \nabla  \lambda_0)\|+C \|(I-{\cal Q}_h)(c\lambda_0)\|\\
\leq &C h\Big(\sum_{T\in {\cal T}_h}\|(\bbeta-\overline{\bbeta})\nabla \lambda_0\|_{1,T}^2 + \|c\lambda_0\|_{1,T}^2\Big)^{\frac{1}{2}} \\
\leq &C h\Big(\sum_{T\in {\cal T}_h}\|\nabla \lambda_0\|_{0,T}^2 + \|\lambda_0\|_{1,T}^2\Big)^{\frac{1}{2}} \\
\leq &Ch \left(\sum_{i=1}^J\|\lambda_0\|^2_{1, \Omega_i}\right)^{\frac{1}{2}},
\end{split}
\end{equation*}
where $\overline{\bbeta}$ is the average of $\bbeta$ on the element $T\in {\cal T}_h$ satisfying the estimate $\|\bbeta-\overline{\bbeta}\|_{L^{\infty}(T)}\leq C h_T$. Thus, we have
$$
(1-Ch^2)\sum_{i=1}^J\|\lambda_0\|^2_{1, \Omega_i} \leq 0,
$$
which leads to $\lambda_0 \equiv 0$ in $\Omega_i,\ i=1, \cdots, J$, provided that the mesh size $h$ is sufficiently small such that $1-Ch^2>0$. This shows that $\lambda_0 \equiv 0$ in $\Omega$, and furthermore, $\lambda_b\equiv 0$ from the fact that $\lambda_b=\lambda_0$ on each  $\partial T$. This completes the proof of the theorem.
\end{proof}

\section{Error Equations}
Let $\lambda$ be the exact solution of the first-order linear convection problem (\ref{model}) and $(\lambda_h, u_h) \in W_{h} \times M_h$ be its numerical approximation arising from the scheme (\ref{al-general})-(\ref{al-general-2}). We define the following two error functions
\begin{align}\label{error}
\epsilon_h&=\lambda_h-Q_h\lambda, \\
e_h&=u_h-{\cal Q} _hu=u_h. \label{error-2}
\end{align}
Note that the exact solution to the dual equation is the trivial function $u=0$.

\begin{lemma}\label{errorequa}
The error functions $\epsilon_h$  and $e_h$ given in (\ref{error})-(\ref{error-2}) satisfy the following error equations:
\begin{eqnarray}\label{sehv}
s(\epsilon_h , \sigma)+b( \sigma, e_h)&=& \ell_{\lambda}(\sigma),\qquad \forall\sigma\in W^{0}_{h },\\
-\tau_2 \sum_{T\in {\cal T}_h}h_T^2 (e_h, v)_T+b(\epsilon_h, v)&=&\zeta_{\lambda}(v),\qquad \forall v\in M_h. \label{sehv2}
\end{eqnarray}
\end{lemma}
Here,
\begin{equation}\label{sigml}
\begin{split}
\ell_{\lambda}(\sigma)= &\sum_{T\in {\cal T}_h} \tau_1 (\bbeta \cdot \nabla (\lambda-Q_0\lambda)-c(\lambda-Q_0\lambda), \bbeta \cdot \nabla \sigma_0-c\sigma_0)_T    \\&-  h_T^{-1}\langle Q_0\lambda-Q_b \lambda, \sigma_0-\sigma_b\rangle_{\partial T},
\end{split}
\end{equation}
\begin{equation}\label{lv}
\qquad \zeta_{\lambda}(v) =\sum_{T\in {\cal T}_h} (\bbeta \cdot (I-{\cal Q}_h)\nabla \lambda-c(\lambda-Q_0\lambda),  v)_T.
\end{equation}

\begin{proof}
From (\ref{al-general-2}) and the commutative property (\ref{l}) we have
\begin{equation*}
\begin{split}
&\quad -\tau_2 \sum_{T\in {\cal T}_h} h_T^2(u_h-{\cal Q}_h u, v)_T+b(\lambda_h-Q_h\lambda, v)\\
& = (f, v)-b(Q_h\lambda, v)\\
& = (f, v)- \sum_{T\in {\cal T}_h}(\bbeta\cdot \nabla_w Q_h\lambda-cQ_0\lambda, v)_T \\
& = (\bbeta \cdot \nabla \lambda-c\lambda, v) - \sum_{T\in {\cal T}_h}(\bbeta\cdot {\cal Q}_h\nabla \lambda-cQ_0\lambda, v)_T\\
& = \sum_{T\in {\cal T}_h} (\bbeta \cdot (I-{\cal Q}_h)\nabla \lambda-c(\lambda-Q_0\lambda),  v)_T,
\end{split}
\end{equation*}
where we used the first equation in (\ref{model}), which gives (\ref{sehv2}). To derive \eqref{sehv}, we subtract $s(Q_h\lambda, \sigma)$ from both sides of (\ref{al-general}) to obtain
\begin{equation*}
\begin{split}
& s( \lambda_h-Q_h\lambda, \sigma)+b(\sigma, u_h-{\cal Q}_hu)  \\
=&  \sum_{T\in {\cal T}_h} \tau_1   (f, \bbeta \cdot \nabla \sigma_0-c\sigma_0)_T-s(Q_h\lambda, \sigma) \\
 =&  \sum_{T\in {\cal T}_h} \tau_1 (\bbeta \cdot \nabla \lambda-c\lambda, \bbeta \cdot \nabla \sigma_0-c\sigma_0)_T
  -  h_T^{-1}\langle Q_0\lambda-Q_b \lambda, \sigma_0-\sigma_b\rangle_{\partial T} \\
  &-\tau_1 (\bbeta \cdot \nabla Q_0\lambda-cQ_0\lambda, \bbeta \cdot \nabla  \sigma_0-c\sigma_0)_T \\
                = & \sum_{T\in {\cal T}_h} \tau_1 (\bbeta \cdot \nabla (\lambda-Q_0\lambda)-c(\lambda-Q_0\lambda), \bbeta \cdot \nabla \sigma_0-c\sigma_0)_T    \\
                &-  h_T^{-1}\langle Q_0\lambda-Q_b \lambda, \sigma_0-\sigma_b\rangle_{\partial T}, \\
  \end{split}
\end{equation*}
which verifies the error equation \eqref{sehv}. This completes the proof of the lemma.
\end{proof}

\section{Error Estimates}
We first introduce a scaled $L^2$ norm in the finite element space $M_h$ as follows:
\begin{equation}\label{mhnorm}
\3bar v\3bar _{M_h}= \Big( \tau_2 \sum_{T\in {\cal T}_h}    h_T^2\| v\|_T^2  \Big)^{\frac{1}{2}}, \qquad v\in M_h,
\end{equation}
where $\tau_2 >0$ is a given parameter. Next, we introduce a semi-norm in the weak finite element space $W_h$:
\begin{equation}\label{whnorm}
\3bar \lambda \3bar_{W_h}= \Big(\sum_{T\in {\cal T}_h} h_T^{-1} \| \lambda_0-\lambda _b \|_{\partial T}^2+ \tau_1   \| \bbeta \cdot \nabla \lambda_0-c\lambda_0 \|_T^2\Big)^{\frac{1}{2}},
\end{equation}
where $\tau_1 \geq 0$ is another parameter.

\begin{lemma} Assume that the solution to the first-order linear convection problem in the non-divergence form \eqref{model} is unique. Then the seminorm $\3bar\cdot\3bar_{W_h}$ given in (\ref{whnorm}) defines a norm in the linear space $W_{h}^{0}$ for any given $\tau_1>0$.
\end{lemma}
\begin{proof} We shall only verify the positivity property for $\3bar \cdot  \3bar _{W_h}$.  To this end, assume $\3bar \lambda \3bar_{W_h}=0$ for some $\lambda =\{\lambda_0, \lambda_b\} \in W_{h}^{0}$. Since $\tau_1>0$, then from (\ref{whnorm}) we have $\lambda_0=\lambda_b$ on $\partial T$ and $\bbeta \cdot \nabla \lambda_0-c\lambda_0=0$ on any $T\in\T_h$. This implies $ \lambda_0\in C^0(\Omega)$ and $\bbeta \cdot \nabla  \lambda_0-c\lambda_0=0$ in $\Omega$. Thus, from $\lambda\in W_{h}^{0}$ and the solution uniqueness for the first-order linear convection problem (\ref{model}) we obtain $\lambda_0 \equiv 0$ and furthermore,  $\lambda_b=\lambda_0=0$. This completes the proof of the lemma.
\end{proof}

 Recall that ${\cal T}_h$ is a shape-regular finite element partition of the domain $\Omega$. Thus, for any $T\in {\cal T}_h$ and $\phi\in H^1(T)$,
 the following trace inequality holds true \cite{wy3655}:
\begin{equation}\label{tracein}
 \|\phi\|^2_{\partial T} \leq C(h_T^{-1}\|\phi\|_T^2+h_T \|\nabla \phi\|_T^2).
\end{equation}
If $\phi$ is a polynomial on the element $T\in {\cal T}_h$,  the following trace inequality holds true \cite{wy3655}; i.e.,
\begin{equation}\label{trace}
\|\phi\|^2_{\partial T} \leq Ch_T^{-1}\|\phi\|_T^2.
\end{equation}

\begin{lemma}\cite{wy3655}
Let ${\cal T}_h$ be a finite element partition of the domain $\Omega$ satisfying the shape regular assumption as specified in \cite{wy3655}. For any $0\leq s \leq 1$ and $0\leq m \leq k$, there holds
\begin{eqnarray}\label{error2}
 \sum_{T\in {\cal T}_h}h_T^{2s}\|u-{\cal Q}_hu\|^2_{s,T}&\leq& C h^{2m}\|u\|^2_{m},\\
\label{error3}
\sum_{T\in {\cal T}_h}h_T^{2s}\|\lambda- Q _0\lambda\|^2_{s,T}&\leq& C h^{2m+2}\|\lambda\|^2_{m+1}.
\end{eqnarray}
 \end{lemma}

\begin{theorem} \label{theoestimate}  Let $\lambda$ and $(\lambda_h; u_h) \in W_h\times  M_h$ be the exact solution of the first-order linear convection problem (\ref{model}) and the primal-dual weak Galerkin solution arising from the numerical scheme (\ref{al-general})-(\ref{al-general-2}), respectively. Assume that the exact
solution $\lambda$ is sufficiently regular such that $\lambda \in \oplus_{i=1}^J H^{k+1}(\Omega_i)$ where $\{\Omega_i\}_{i=1}^J$ is a non-overlapping partition of the domain $\Omega$. There exists a constant $C$ such that
 \begin{equation}\label{erres}
  \3bar\epsilon_h\3bar_{W_h}+\3bar e_h \3bar_{M_h} \leq C (1+\tau_2^{-\frac{1}{2}})h^{k}\|\lambda\|_{k+1}.
\end{equation}
\end{theorem}

\begin{proof} By setting $\sigma=\epsilon_h$ in the error equation (\ref{sehv}) and $v=e_h$ in (\ref{sehv2}), we have from the resulting (\ref{sehv}) and (\ref{sehv2}) that
$$
\tau_2\sum_{T\in {\cal T}_h}h_T^2 (e_h, e_h)_T+s(\epsilon_h, \epsilon_h)= \ell_{\lambda}(\epsilon_h)-\zeta_{\lambda}(e_h),
$$
which gives
\begin{equation}\label{aij}
\begin{split}
\3bar e_h \3bar_{M_h}^2+\3bar\epsilon_h\3bar_{W_h}^2
\leq & |\ell_{\lambda}(\epsilon_h)|+|\zeta_{\lambda}(e_h)| = I_1+I_2.
 \end{split}
 \end{equation}

We shall estimate the two terms $I_1$ and $I_2$ in (\ref{aij}). For the term $I_1$, it follows from the Cauchy-Schwarz inequality, the triangle inequality, (\ref{sigml}), the trace inequality (\ref{tracein}), and the estimate (\ref{error3}) with $m=k$ that
 \begin{equation}\label{term2}
\begin{split}
I_1=&\Big| \sum_{T\in {\cal T}_h} \tau_1 (\bbeta \cdot \nabla (\lambda-Q_0\lambda)-c(\lambda-Q_0\lambda), \bbeta \cdot \nabla \epsilon_0-c\epsilon_0)_T \\&- h_T^{-1}\langle Q_0\lambda-Q_b \lambda, \epsilon_0-\epsilon_b\rangle_{\partial T}\Big|\\
\leq &\left(\Big(\sum_{T\in {\cal
T}_h} \tau_1\|\bbeta \cdot \nabla (\lambda-Q_0\lambda)\|^2_{T}\Big)^{\frac{1}{2}}+\Big(\sum_{T\in {\cal
T}_h}\tau_1 \|c(\lambda-Q_0\lambda)\|^2_{T}\Big)^{\frac{1}{2}}\right) \\
&\cdot \Big(\sum_{T\in {\cal T}_h} \tau_1 \| \bbeta \cdot \nabla \epsilon_0-c\epsilon_0\|^2_{T}\Big)^{\frac{1}{2}}\\& +\Big(\sum_{T\in {\cal
T}_h} h_T^{-1} \|\epsilon_0-\epsilon_b\|^2_{\partial
T}\Big)^{\frac{1}{2}} \Big(\sum_{T\in {\cal T}_h} h_T^{-1} \| Q_0\lambda- Q_b \lambda\|^2_{\partial T}\Big)^{\frac{1}{2}}\\
\leq & \3bar\epsilon_h \3bar_{W_h} (Ch^{k}\|\lambda\|_{k+1}+Ch^{k+1}\|\lambda\|_{k+1}) \\&+ \3bar\epsilon_h \3bar_{W_h} \Big(\sum_{T\in {\cal T}_h} h_T^{-1}\| Q_0\lambda-  \lambda\|^2_{\partial T}\Big)^{\frac{1}{2}}\\
\leq &\3bar\epsilon_h \3bar_{W_h} \Big(Ch^{k}\|\lambda\|_{k+1}+C(\sum_{T\in {\cal T}_h}h_T^{-2} \| Q_0\lambda-  \lambda\|^2_{ T}+ \|Q_0\lambda- \lambda\|^2_{1,T})^{\frac{1}{2}}\Big)\\
\leq & C h^{k}\3bar\epsilon_h \3bar _{W_h} \|\lambda\|_{k+1}.
\end{split}
\end{equation}

As to term $I_2$, we use the orthogonality property of ${\cal Q}_h$ to obtain
\begin{equation*}
\begin{split}
I_2=& \Big|\sum_{T\in {\cal T}_h} (\bbeta \cdot (I-{\cal Q}_h)\nabla \lambda-c(\lambda-Q_0\lambda),  e_h)_T\Big| \\
\le & \Big|\sum_{T\in {\cal T}_h} (\bbeta \cdot (I-{\cal Q}_h)\nabla\lambda, e_h)_T\Big| +\Big|\sum_{T\in {\cal T}_h}  (c(\lambda-Q_0\lambda),  e_h)_T\Big|\\
=&\Big| \sum_{T\in {\cal T}_h} ((I-{\cal Q}_h)\nabla \lambda,  (I-{\cal Q}_h)(\bbeta-\overline{\bbeta}) e_h)_{T}\Big|\\&+\Big|\sum_{T\in {\cal T}_h}(\lambda-Q_0\lambda, c e_h)_T\Big|.
\end{split}
\end{equation*}
Next, from the Cauchy-Schwarz inequality, (\ref{lv}), the triangle inequality, the estimate (\ref{error2}) with $m=k$ and $m=1$, the estimate (\ref{error3}) with $m=k$, and the inverse inequality we obtain

\begin{equation}\label{term3}
\begin{split}
|I_2| \leq & \Big(\sum_{T\in {\cal
T}_h} \| (I-{\cal Q}_h)\nabla \lambda\|^2_{T}\Big)^{\frac{1}{2}} \Big(\sum_{T\in {\cal T}_h}  \| (I-{\cal Q}_h)(\bbeta-\overline{\bbeta}) e_h\|^2_{T}\Big)^{\frac{1}{2}}\\
&+\|c\|_{L^{\infty}(\Omega)} \Big(\sum_{T\in {\cal
T}_h}  \tau_2^{-1}h_T^{-2}\|Q_0\lambda-\lambda\|^2_{T}\Big)^{\frac{1}{2}} \Big(\sum_{T\in {\cal T}_h}   \tau_2 h^2_T\|  e_h\|^2_{T}\Big)^{\frac{1}{2}}\\
\leq & C h^{k}\|\lambda\|_{k+1} \Big(\sum_{T\in {\cal T}_h}  h_T^2\|(\bbeta-\overline{\bbeta}) e_h\|^2_{1, T}\Big)^{\frac{1}{2}}+C\tau_2^{-\frac{1}{2}}h^k\|\lambda\|_{k+1}\3bar e_h\3bar_{M_h}\\
\leq & C h^{k}\|\lambda\|_{k+1} \Big(\sum_{T\in {\cal T}_h} h_T^2(  \|e_h \nabla \bbeta\|^2_{T}+  \|(\bbeta-\overline{\bbeta}) \cdot \nabla e_h\|^2_{T})\Big)^{\frac{1}{2}}\\
&+C\tau_2^{-\frac{1}{2}}h^k\|\lambda\|_{k+1}\3bar e_h\3bar_{M_h}\\
\leq & C \tau_2^{-\frac{1}{2}} h^{k}\|\lambda\|_{k+1} \Big(\sum_{T\in {\cal T}_h} \tau_2 h_T^2( \|e_h\|_{T}^2+h_T^2h_T^{-2} \|e_h\|^2_{T})\Big)^{\frac{1}{2}}\\&+C\tau_2^{-\frac{1}{2}}h^k\|\lambda\|_{k+1}\3bar e_h\3bar_{M_h}\\
\leq & C\tau_2^{-\frac{1}{2}}  h^{k}\|\lambda\|_{k+1} \3bar e_h \3bar_{M_h},\\
\end{split}
\end{equation}
where $\overline{\bbeta}$ is the average of $\bbeta$ on each $T\in {\cal T}_h$ satisfying $\|\bbeta-\overline{\bbeta}\|_{L^\infty(T)} \leq Ch_T$. Combining (\ref{aij}) with (\ref{term2}) and (\ref{term3}) yields the error estimate \eqref{erres}.
This completes the proof of the theorem.
\end{proof}

\section{An Error Estimate in $L^2$}\label{Section:H1L2Error}
For any $\theta\in L^2(\Omega)$, consider the auxiliary problem of seeking an unknown function $w$ satisfying \begin{align}\label{dual1}
-\nabla \cdot (\bbeta w)-c w=& \theta,\qquad \text{in}\ \Omega,\\
w=& 0,\qquad \text{on}\ \Gamma_+, \label{dual2}
\end{align}
where $\Gamma_+=\partial \Omega \setminus \Gamma_-$ is the outflow boundary satisfying $\bbeta\cdot \bn \geq 0$ with $\bn$ being the unit outward normal direction to $\partial \Omega$.
The problem (\ref{dual1})-(\ref{dual2}) is said to satisfy a local $H^1$-regularity assumption if there exists a non-overlapping partition of the domain $\Omega = \bigcup_{i=1}^J \Omega_i$ such that there exists a solution $w$, $w\in H^1(\Omega_i)$ for $i=1, \cdots, J$, satisfying
\begin{equation}\label{regul}
\Big(\sum_{i=1}^J \|w\|^2_{1,\Omega_i}\Big)^{\frac{1}{2}}\leq C\|\theta\|,
\end{equation}
where $C$ is a generic constant.

\begin{lemma}\label{lemma7.2} For any $\sigma=\{\sigma_0, \sigma_b\}\in W_k(T)$, there holds
\begin{equation}\label{qaij}
\|\nabla_{w} \sigma\|_T^2 \leq C\left(\|\nabla \sigma_0\|_T^2
+ s_T(\sigma, \sigma)\right),
\end{equation}
where $C$ is a constant independent of $T\in\T_h$.
\end{lemma}
\begin{proof}
It follows from (\ref{disgradient*}), the Cauchy-Schwarz inequality, the triangle inequality, and the trace inequality (\ref{trace}) that
\begin{equation*}\label{7.22}
\begin{split}
|(\nabla_{w} \sigma, \boldsymbol{\varphi})_T| =&| (\nabla \sigma_0
,\boldsymbol{\varphi})_T-\langle \sigma_0-\sigma_b, \boldsymbol{\varphi} \cdot \bn \rangle_{\pT}|\\
\leq &| (\nabla \sigma_0
,\boldsymbol{\varphi})_T|+|\langle \sigma_0-\sigma_b, \boldsymbol{\varphi} \cdot \bn \rangle_{\pT}|\\
 \leq & \|\nabla \sigma_0\|_T
\|\boldsymbol{\varphi}\|_T + \|\sigma_0-\sigma_b\|_\pT \|\boldsymbol{\varphi} \cdot \bn\|_{\pT} \\
\leq & \left(\|\nabla \sigma_0\|_T +
Ch_T^{-\frac12}\|\sigma_b-\sigma_0\|_\pT \right)\|\boldsymbol{\varphi}\|_T,
\end{split}
\end{equation*}
for any $\boldsymbol{\varphi}\in [P_{k-1}(T)]^d$, which implies
$$
\|\nabla_{w} \sigma\|_T^2\leq C\left( \|\nabla \sigma_0\|_T^2 +
h_T^{-1}\|\sigma_b-\sigma_0\|_\pT^2 \right).
$$
This completes the proof of the lemma.
\end{proof}

\begin{lemma}\label{Lemma:TechnicalEquality}
For any $\sigma=\{\sigma_0, \sigma_b\}\in W_{h}^{0}$, the following identity holds true
\begin{equation}\label{2.14:800}
\begin{split}
(\sigma_0, \theta) =\sum_{T\in{\cal T}_h} (\bbeta \cdot \nabla_w \sigma-c \sigma_0, w)_T+\langle \sigma_0-\sigma_b, ({\cal Q}_h-I)(\bbeta w) \cdot \bn\rangle_{\pT}.
\end{split}
\end{equation}

\end{lemma}

\begin{proof} From testing (\ref{dual1}) with $\sigma_0$ on each element $T\in\T_h$, and then using the integration by parts we have
\begin{equation}\label{2.11}
\begin{split}
(\sigma_0, \theta)
=&\sum_{T\in{\cal T}_h}(-\nabla \cdot (\bbeta w)-c w, \sigma_0)_T\\
=&\sum_{T\in{\cal T}_h}(w, \bbeta \cdot \nabla \sigma_0 -c \sigma_0)_T- \langle \bbeta w \cdot \bn, \sigma_0\rangle_{\partial T} \\
=&\sum_{T\in{\cal T}_h}(w, \bbeta \cdot \nabla \sigma_0 -c \sigma_0)_T- \langle \bbeta w \cdot \bn, \sigma_0-\sigma_b\rangle_{\partial T},
\end{split}
\end{equation}
where we have used the homogeneous boundary condition (\ref{dual2})
and $\sigma_b=0$ on $\Gamma_-$ in the last line.

Next, by setting $\boldsymbol{\varphi}={\cal Q}_h(\bbeta w)$ in  (\ref{disgradient*}), we have
\begin{align*}
(\nabla_w \sigma, {\cal Q}_h(\bbeta w))_T = &(\nabla
\sigma_0,{\cal Q}_h(\bbeta w))_T - \langle \sigma_0-\sigma_b, {\cal Q}_h(\bbeta w) \cdot \bn\rangle_{\pT} \\
=&(\nabla
\sigma_0, \bbeta w)_T - \langle \sigma_0-\sigma_b, {\cal Q}_h(\bbeta w) \cdot \bn \rangle_{\pT},
\end{align*}
which leads to
\begin{equation}\label{2.13}
\begin{split}
 (\bbeta \cdot \nabla \sigma_0, w)_T =&(\nabla_w \sigma, {\cal Q}_h(\bbeta w))_T+\langle \sigma_0-\sigma_b, {\cal Q}_h(\bbeta w) \cdot \bn \rangle_{\pT}\\
 =&(\nabla_w \sigma,  \bbeta w)_T+\langle \sigma_0-\sigma_b, {\cal Q}_h(\bbeta w) \cdot \bn \rangle_{\pT}.
\end{split}
\end{equation}
Substituting (\ref{2.13}) into (\ref{2.11}) yields
\begin{equation*}\label{2.14}
\begin{split}
(\sigma_0, \theta) =&\sum_{T\in{\cal T}_h} (\bbeta \cdot \nabla_w \sigma-c \sigma_0, w)_T+\langle \sigma_0-\sigma_b, ({\cal Q}_h-I)(\bbeta w) \cdot \bn\rangle_{\pT},
\end{split}
\end{equation*}
which completes the proof of the lemma.
\end{proof}

The following lemma is devoted to an estimate for the second term on the right-hand side of (\ref{2.14:800}) based on the local $H^1$-regularity assumption \eqref{regul} for the auxiliary problem (\ref{dual1})-(\ref{dual2}).

\begin{lemma}\label{Lemma:TechnicalEstimates:01} Assume that the auxiliary problem (\ref{dual1})-(\ref{dual2}) satisfies the local $H^1$-regularity assumption \eqref{regul}. Then, for any $\sigma\in W_{h}^{0}$, there holds
\begin{eqnarray}\label{2.14.100:10}
\left|\sum_{T\in{\cal T}_h}\langle \sigma_0-\sigma_b, ({\cal Q}_h-I)(\bbeta w) \cdot \bn \rangle_{\pT}
\right|&\leq & Ch  \|\theta\|
\3bar \sigma\3bar_{W_h}.
\end{eqnarray}
\end{lemma}

\begin{proof} From the Cauchy-Schwarz inequality, the trace
inequality (\ref{tracein}), the local $H^1$-regularity \eqref{regul},  and the estimate (\ref{error2}) with $m=1$, we have
\begin{equation*}\label{2.14.100}
\begin{split}
& \left|\sum_{T\in{\cal T}_h}\langle \sigma_0-\sigma_b, ({\cal Q}_h-I)(\bbeta w) \cdot \bn \rangle_{\pT}
\right| \\
\leq & \Big(\sum_{T\in{\cal T}_h}  h_T^{-1}\|\sigma_0-\sigma_b\|^2_\pT\Big)^{\frac{1}{2}} \Big(\sum_{T\in{\cal T}_h} h_T\|({\cal Q}_h-I)(\bbeta w) \cdot \bn\|^2_\pT\Big)^{\frac{1}{2}}\\
\leq & \ C\3bar \sigma\3bar_{W_h} \Big(\sum_{T\in{\cal T}_h}  \|({\cal Q}_h-I)(\bbeta w)\|^2_T+h_T^2\|({\cal Q}_h-I)(\bbeta w)\|^2_{1, T}\Big)^{\frac{1}{2}} \\
\leq & \ C\3bar \sigma\3bar_{W_h}  \left( \sum_{i=1}^J h^2\|w\|_{1, \Omega_i}^2\right)^{\frac{1}{2}}   \\
\leq &  Ch \|\theta\| \3bar \sigma\3bar_{W_h},
\end{split}
\end{equation*}
which completes the proof of the lemma.
\end{proof}

The following is an error estimate in the usual $L^2$ norm for the first component $\lambda_0$ of the primal variable $\lambda_h$.

\begin{theorem}\label{Thm:H1errorestimate} Let $\lambda_h=\{\lambda_0, \lambda_b\}\in W_{h}$ be the primal-dual weak Galerkin solution arising from the PD-WG scheme (\ref{al-general})-(\ref{al-general-2}) with $u_h\in M_h$ being the numerical Lagrangian multiplier. Assume that the exact solution $\lambda$ of the first-order linear convection model problem (\ref{model}) is sufficiently regular such that $\lambda\in H^{k+1}(\Omega)$. Under the local $H^1$-regularity assumption  \eqref{regul} for the auxiliary problem (\ref{dual1})-(\ref{dual2}), there holds
\begin{equation}\label{e0-H1}
\|\epsilon_0\|_0
\leq C(1+(h+1+\tau_2^{\frac{1}{2}})(1+\tau_2^{-\frac{1}{2}}))  h^{k+1} \|\lambda\|_{k+1},
\end{equation}
provided that the meshsize $h$ is sufficiently small such that $h<h_0$ for a fixed but sufficiently small $h_0$.
\end{theorem}

\begin{proof} Let $w$ be the solution of the auxiliary problem
(\ref{dual1})-(\ref{dual2}) for a given function $\theta$. By setting $\sigma=\epsilon_h \in W_h^{0}$ in Lemma \ref{Lemma:TechnicalEquality} we obtain
\begin{equation}\label{2.14:800:10}
\begin{split}
(\epsilon_0, \theta ) =&\sum_{T\in{\cal T}_h} (\bbeta \cdot \nabla_w \epsilon_h-c \epsilon_0, w)_T+\sum_{T\in{\cal T}_h} \langle \epsilon_0-\epsilon_b, ({\cal Q}_h-I)(\bbeta w) \cdot \bn\rangle_{\pT}\\
= & I_1 + I_2,
\end{split}
\end{equation}
where $I_1$ and $I_2$ are defined accordingly.

For the term $I_2$, we use the estimate in Lemma \ref{Lemma:TechnicalEstimates:01} to obtain
\begin{equation}\label{2.14:800:15}
|I_2| \leq C h \|\theta\| \3bar \epsilon_h\3bar_{W_h}.
\end{equation}

As to the term $I_1$, we use the error equation (\ref{sehv2}) to
obtain
\begin{equation}\label{2.14.120}
\begin{split}
I_1 = & \sum_{T\in{\cal T}_h} (\bbeta \cdot \nabla_w  \epsilon_h- c\epsilon_0, w)_T\\
= & \sum_{T\in{\cal T}_h} (\bbeta \cdot \nabla_w  \epsilon_h-c  \epsilon_0, {\cal Q}_h w)_T+\sum_{T\in{\cal T}_h}  (\bbeta \cdot \nabla_w  \epsilon_h-c  \epsilon_0, (I-{\cal Q}_h) w)_T\\
 = & \sum_{T\in {\cal T}_h} (\bbeta \cdot (I-{\cal Q}_h)\nabla \lambda,  {\cal Q}_h w)_T
 +\sum_{T\in {\cal T}_h} (c(Q_0\lambda-\lambda),  {\cal Q}_h w)_T
 \\
 & +\sum_{T\in {\cal T}_h}  \tau_2h_T^2 (e_h,  {\cal Q}_h w)_T  +\sum_{T\in {\cal T}_h} (\bbeta \cdot \nabla_w  \epsilon_h-c  \epsilon_0, (I-{\cal Q}_h) w)_T \\
=&I_{11}+I_{12}+I_{13}+I_{14},
\end{split}
\end{equation}
where $I_{11}, I_{12}, I_{13}$, and $I_{14}$ are defined accordingly. We shall estimate the four terms $I_{11}, I_{12}, I_{13}$, and $ I_{14}$ in (\ref{2.14.120}) one by one. For the term $I_{11}$, from the Cauchy-Schwarz inequality, the estimate (\ref{error2}) with $m=1$ and $m=k$, and the regularity assumption \eqref{regul}, we have
\begin{equation}\label{EQ:New:2015:800}
\begin{split}
|I_{11}| = &| \sum_{T\in {\cal T}_h} (\bbeta \cdot (I-{\cal Q}_h) \nabla \lambda, {\cal Q}_h w)_T|\\
=&| \sum_{T\in {\cal T}_h} ((I-{\cal Q}_h) \nabla \lambda, (I-{\cal Q}_h) \bbeta \cdot {\cal Q}_h w)_T|\\
 =&\Big(\sum_{T\in {\cal T}_h} \|(I-{\cal Q}_h) \nabla \lambda\|_T^2\Big)^{\frac{1}{2}}\Big(\sum_{T\in {\cal T}_h}\|(I-{\cal Q}_h) \bbeta \cdot {\cal Q}_h w\|_T^2\Big)^{\frac{1}{2}}  \\
\leq &Ch^{k+1}\| \lambda\| _{k+1}\left(\sum_{i=1}^J \|w\|^2_{1, \Omega_i}\right)^{\frac{1}{2}}\\
\leq &Ch^{k+1}\| \lambda\| _{k+1}\|\theta\|.
\end{split}
\end{equation}
For the term $I_{12}$, we use the Cauchy-Schwarz inequality, the estimate (\ref{error3}) with $m=k$, and the regularity assumption \eqref{regul} to obtain
\begin{equation}\label{EQ:New:2015:8}
\begin{split}
|I_{12}|=&| \sum_{T\in {\cal T}_h} (Q_0\lambda-\lambda, c {\cal Q}_h w)_T|\\
  =&\Big(\sum_{T\in {\cal T}_h} \|Q_0\lambda-\lambda\|_T^2\Big)^{\frac{1}{2}}\Big(\sum_{T\in {\cal T}_h}\|c {\cal Q}_h w\|_T^2\Big)^{\frac{1}{2}}  \\
\leq &Ch^{k+1}\| \lambda\| _{k+1}(\sum_{i=1}^J \|w\|^2_{1, \Omega_i})^{\frac{1}{2}}\\
\leq &Ch^{k+1}\| \lambda\| _{k+1}\|\theta\|.
\end{split}
\end{equation}
For the term $I_{13}$, we use the Cauchy-Schwarz inequality, the error estimate (\ref{erres}), and the regularity assumption \eqref{regul} to obtain
\begin{equation}\label{EQ}
\begin{split}
|I_{13}|=&| \sum_{T\in {\cal T}_h} \tau_2 h_T^2(e_h, {\cal Q}_h w)_T|\\
\leq & \Big(\sum_{T\in {\cal T}_h} \tau_2 h_T^2 \|e_h\|_T^2\Big)^{\frac{1}{2}} \Big(\sum_{T\in {\cal T}_h}  \tau_2h_T^2 \| {\cal Q}_h w\|_T^2\Big)^{\frac{1}{2}} \\
\leq & C\tau_2^{\frac{1}{2}}  h\3bar e_h\3bar_{M_h} \|w\|_{0} \\
\leq &C\tau_2^{\frac{1}{2}}(1+\tau_2^{-\frac{1}{2}})h^{k+1}\|\lambda\|_{k+1} \|\theta\|.
\end{split}
\end{equation}
As to the term $I_{14}$, from the Cauchy-Schwarz inequality, (\ref{qaij}), the estimate (\ref{error2}) with $m=1$, the inverse inequality, and the regularity assumption \eqref{regul} we have
\begin{equation}\label{EQ:New:2015:810}
\begin{split}
|I_{14}|=&| \sum_{T\in {\cal T}_h} (\bbeta \cdot \nabla_w \epsilon_h-c \epsilon_0, (I-{\cal Q}_h) w)_T|\\
 \leq &| \sum_{T\in {\cal T}_h} ((\bbeta-\overline{\bbeta}) \cdot \nabla_w \epsilon_h-c \epsilon_0, (I-{\cal Q}_h) w)_T|\\
 \leq &\left( \Big(\sum_{T\in {\cal T}_h} \|\bbeta-\overline{\bbeta}\|^2_{L^{\infty}(T)}\|\nabla_w \epsilon_h\|_T^2\Big)^{\frac{1}{2}}+\|c\|_{L^{\infty}(\Omega)} \Big(\sum_{T\in {\cal T}_h}   \|\epsilon_0\|_T^2\Big)^{\frac{1}{2}}\right)\\
 &\cdot \Big(\sum_{T\in {\cal T}_h}\| (I-{\cal Q}_h) w\|_T^2\Big)^{\frac{1}{2}}\\
  \leq &\left(\Big(\sum_{T\in {\cal T}_h} h_T^2 \|\nabla_w \epsilon_h\|_T^2\Big)^{\frac{1}{2}}+C \Big(\sum_{T\in {\cal T}_h}  \|\epsilon_0\|_T^2\Big)^{\frac{1}{2}}\right)Ch \left(\sum_{i=1}^J \|w\|^2_{1, \Omega_i}\right)^{\frac{1}{2}}\\
\leq & C h  \left(\Big(C\sum_{T\in {\cal T}_h}h_T^2 \|\nabla \epsilon_0\|_T^2+ h_T^2s_T(\epsilon_h, \epsilon_h)\Big)^{\frac{1}{2}}+C \Big(\sum_{T\in {\cal T}_h}  \|\epsilon_0\|_T^2\Big)^{\frac{1}{2}}\right)\|\theta\| \\
 \leq & Ch \left(C\| \epsilon_0\| +Ch\3bar\epsilon_h\3bar_{W_h}\right)\|\theta\| ,
\end{split}
\end{equation}
where $\overline{\bbeta}$ is the average of $\bbeta$ on each element $T\in{\cal T}_h $ satisfying $\|\bbeta-\overline{\bbeta}\|_{L^{\infty}(T)}\leq Ch_T$. By substituting (\ref{EQ:New:2015:800}) - (\ref{EQ:New:2015:810}) into (\ref{2.14.120}), we obtain the following estimate for the term $I_1$:
\begin{equation}\label{EQ:New:2015:820}
\begin{split}
&|I_1| \leq  C h  \Big(C\| \epsilon_0\| +h\3bar\epsilon_h\3bar_{W_h} +(1+\tau_2^{\frac{1}{2}}(1+\tau_2^{-\frac{1}{2}}))h^{k}\|\lambda\|_{k+1}  \Big)\|\theta\|\\
\leq & C h \Big(C\| \epsilon_0\|+ (1+\tau_2^{-\frac{1}{2}}) h^{k+1}\|\lambda\|_{k+1}    +(1+\tau_2^{\frac{1}{2}}(1+\tau_2^{-\frac{1}{2}}))h^{k}\|\lambda\|_{k+1}  \Big)\|\theta\|,
\end{split}
\end{equation}
where we have used the error estimate (\ref{erres}).

Substituting (\ref{EQ:New:2015:820}) and (\ref{2.14:800:15}) into (\ref{2.14:800:10}) yields
\begin{equation*}
\begin{split}
|(\epsilon_0, \theta )| \leq & C h \Big(C\| \epsilon_0\|+ (1+\tau_2^{-\frac{1}{2}}) h^{k+1}\|\lambda\|_{k+1} +(1+\tau_2^{\frac{1}{2}}(1+\tau_2^{-\frac{1}{2}}))h^{k}\|\lambda\|_{k+1}  \Big)\|\theta\| \\
&+ C(1+\tau_2^{-\frac{1}{2}}) \|\theta\|  h^{k+1}\|\lambda\|_{k+1},
\end{split}
\end{equation*}
where we used the error estimate (\ref{erres}). Since the set of all such $\theta$ is dense in $L^2(\Omega)$, the
above inequality leads to
\begin{equation*}
\begin{split}
\|\epsilon_0\|&\leq C h \Big(C\| \epsilon_0\|+ (1+\tau_2^{-\frac{1}{2}}) h^{k+1}\|\lambda\|_{k+1} +(1+\tau_2^{\frac{1}{2}}(1+\tau_2^{-\frac{1}{2}}))h^{k}\|\lambda\|_{k+1}  \Big) \\
&+ C(1+\tau_2^{-\frac{1}{2}}) h^{k+1}\|\lambda\|_{k+1},
\end{split}
\end{equation*}
and furthermore
\begin{equation*}\label{EQ:New:2015:820:100}
(1- Ch)\| \epsilon_0\|\leq C(1+(h+1+\tau_2^{\frac{1}{2}})(1+\tau_2^{-\frac{1}{2}}))  h^{k+1}\|\lambda\|_{k+1}.
\end{equation*}
This completes the proof of the theorem provided that the meshsize $h$ is sufficiently small  such that $1- Ch>\frac12$.
\end{proof}

To establish an error estimate for the boundary component
$\lambda_b$, we introduce the following norm
\begin{equation*}\label{EQ:eb-eg-L2norm}
\|\epsilon_b\|_{0,*}:=\Big(\sum_{T\in {\cal T}_h} h_T\|\epsilon_b\|_{\partial
T}^2\Big)^{\frac{1}{2}}.
\end{equation*}

\begin{theorem}\label{Thm:L2errorestimate-ub}
Under the assumptions of Theorem \ref{Thm:H1errorestimate}, there
holds
$$
\|\epsilon_b\|_{0,*} \leq C(1+(h+1+\tau_2^{\frac{1}{2}})(1+\tau_2^{-\frac{1}{2}}))  h^{k+1} \|\lambda\|_{k+1}.
$$
\end{theorem}

\begin{proof} On each element $T\in\T_h$, from the triangle inequality we have
$$
\|\epsilon_b\|_{\pT} \leq \|\epsilon_0\|_{\pT} + \|\epsilon_b-\epsilon_0\|_\pT.
$$
Thus, it follows from the trace inequality (\ref{trace}) that
\begin{equation*}
\begin{split}
\sum_{T\in\T_h} h_T \|\epsilon_b\|_{\pT}^2 & \leq C \sum_{T\in\T_h}
h_T\|\epsilon_0\|_{\pT}^2 + C h^2 \sum_{T\in\T_h}
h_T^{-1}\|\epsilon_b-\epsilon_0\|_\pT^2\\
& \leq C (\|\epsilon_0\|^2 + h^2 \3bar \epsilon_h\3bar_{W_h}^2),
\end{split}
\end{equation*}
which, together with the error estimates (\ref{erres}) and
(\ref{e0-H1}), completes the proof of the theorem.
\end{proof}

\section{Numerical Experiments}\label{Section:NE}
The primal dual weak Galerkin finite element scheme (\ref{al-general})-(\ref{al-general-2}) has been implemented for the case of $k=1$ and $k=2$ respectively. The objective of this section is to report some of the computational results.

For $k=1$, the numerical primal variable $\lambda_h$ and the dual variable $u_h$ are obtained from the following finite element spaces:
\begin{equation*}
\begin{split}
&W_h^{(1)}=\{\lambda_h=\{\lambda_0,\lambda_b\}: \ \lambda_0\in P_1(T), \lambda_b\in P_1(e), e\subset\pT, T\in {\cal T}_h\},\\
&M_{h}^{(1)}=\{u_h: u_h|_T \in P_{0}(T),\ \forall T\in {\cal T}_h\}.
\end{split}
\end{equation*}
For convenience, the corresponding finite element shall be referred to as the $C^{-1}- P_1(T)/P_1(\partial T)/P_{0}(T)$ element.

For $k=2$, the finite element spaces for $\lambda_h$ and $u_h$ are given as follows:
\begin{equation*}
\begin{split}
&W_h^{(2)}=\{\lambda_h=\{\lambda_0,\lambda_b\}: \ \lambda_0\in P_2(T), \lambda_b\in P_2(e), e\subset\pT, T\in {\cal T}_h\},\\
&M_{h}^{(2)}=\{u_h: u_h|_T \in P_{1}(T),\ \forall T\in {\cal T}_h\}.
\end{split}
\end{equation*}
The corresponding finite element shall be referred to as $C^{-1}- P_2(T)/P_2(\partial T)/P_{1}(T)$ element.

The numerical solution $\lambda_h=\{\lambda_0, \lambda_b\}\in W_h^{(k)}$ and $u_h \in M_{h}^{(k)}$ obtained from the PD-WG scheme (\ref{al-general})-(\ref{al-general-2}) is compared with the $L_2$ projection of the exact solution $\lambda$ and $u$. Note that $u_h$ approximates the trivial exact solution $u=0$. The error functions are respectively defined as
$$
\epsilon_0=\lambda_0-Q_0 \lambda,\quad \epsilon_b=\lambda_b -Q_b \lambda,  \quad \text{and}\quad e_h=u_h-{\cal Q}_h u=u_h.
$$
The following $L^2$ norms are employed to measure the error:
$$
 \|\epsilon_0\|= \left(\sum_{T\in {\cal T}_h} \int_T \epsilon_0^2 dT\right)^{\frac{1}{2}}, \qquad \|\epsilon_b\|=\left(\sum_{T\in {\cal T}_h}h_T \int_{\partial T} \epsilon_b^2 ds\right)^{\frac{1}{2}},
 $$
 $$
 \|e_h\|= \left(\sum_{T\in {\cal T}_h} \int_T e_h^2 dT\right)^{\frac{1}{2}}.
$$

The numerical experiments are conducted on several polygonal domains $\Omega_i$; some are convex and the others are non-convex. The domain $\Omega_1$ is convex and is chosen as the unit square $\Omega_1=(0,1)^2$. The domain $\Omega_2$ is L-shaped with vertices $A_1=(0, 0)$, $A_2=(1, 0)$, $A_3=(1, 0.5)$, $A_4=(0.5, 0.5)$, $A_5=(0.5, 1)$, and $A_6=(0, 1)$. The third domain is a cracked unit square given by $\Omega_3=(0, 1)^2\setminus(0.5, 1)*\{0.5\}$ with a crack along the edge $(0.5,1)\times\{0.5\}$. The inflow boundary $\Gamma_-$ is determined by the condition of $\bbeta \cdot \bn < 0$, where $\bn$ is the unit outward normal direction to $\partial \Omega$. The right-hand side function $f$ and the inflow Dirichlet boundary data $g$ are chosen to match the exact solution $\lambda$ if it is known for the test problem.

Uniform triangular and/or rectangular finite element partitions are considered in the numerical tests. The uniform triangulations are generated through a successive uniform refinement of a coarse triangulation of the domain by dividing each coarse triangular element into four congruent sub-triangles by connecting the mid-points on the three edges of the triangular element. The rectangular finite element partitions are generated through a successive uniform refinement of a coarse $3\times 2$ rectangular partition of the domain by dividing each coarse rectangular element into four congruent sub-rectangles by connecting the mid-points on the two parallel edges of the rectangular element respectively.

\subsection{Constant-valued convection vector $\bbeta$}
This test problem assumes the domain $\Omega_1$, the exact solution is given by $\lambda=\cos(x)\cos(y)$, the convection tensor is $\bbeta=[1, 1]$, and the reaction coefficient is $c=1$. Tables \ref{NE:TRI:Case1-1-1-1}-\ref{NE:REC:Case1-1-0-0} illustrate the numerical performance of the $C^{-1}-P_1(T)/P_1(\partial T)/P_0(T)$ element when triangular and rectangular partitions are employed. The numerical results in Tables \ref{NE:TRI:Case1-1-1-1}-\ref{NE:TRI:Case1-1-0-1} suggest that the convergence rates for $\epsilon_0$ and $\epsilon_b$ in the discrete $L^2$ norm are a bit higher than the optimal order of ${\cal O}(h^2)$ with the parameter value $(\tau_1, \tau_2)=(1, 1)$ and $(\tau_1, \tau_2)=(0, 1)$ respectively. The numerical results clearly outperform what the theory predicts. Tables \ref{NE:REC:Case1-1-1-1}-\ref{NE:REC:Case1-1-0-1} indicate that the convergence rates for $\epsilon_0$ and $\epsilon_b$ in the discrete $L^2$ norm are of order ${\cal O}(h^2)$ for the parameters $(\tau_1, \tau_2)=(1, 1)$ and $(\tau_1, \tau_2)=(0, 1)$ respectively, which are in good consistency with the theory.
Tables \ref{NE:TRI:Case1-1-0-0} and \ref{NE:REC:Case1-1-0-0}  illustrate the convergence rates for $\epsilon_0$ and $\epsilon_b$ in the discrete $L^2$ norm seem to reach an optimal order of ${\cal O}(h^2)$ for the parameters $(\tau_1, \tau_2)=(0, 0)$. Note that  the theory established in Theorems \ref{Thm:H1errorestimate}-\ref{Thm:L2errorestimate-ub} is not applicable to the case $\tau_2=0$.

\begin{table}[H]
\begin{center}
\caption{Numerical rates of convergence for the $C^{-1}-P_1(T)/P_1(\partial T)/P_0(T)$ element with the exact solution $\lambda=\cos(x)\cos(y)$ on the unit square domain $\Omega_1$; uniform triangular partitions; convection vector $\bbeta=[1, 1]$; reaction coefficient $c=1$; and the parameters $(\tau_1,\tau_2)=(1, 1)$.}\label{NE:TRI:Case1-1-1-1}
\begin{tabular}{|c|c|c|c|c|c|c|}
\hline
$1/h$        & $\|\epsilon_0\| $ &  order&  $\|\epsilon_b\| $ &  order&  $\|e_h\|$ &order
\\
\hline
1&        2.0109E-01  &&	         3.9891E-01 &                  & 3.0140E-03 	&
\\
\hline
2&        6.2996E-02 &	1.6745      &	1.1303E-01 & 1.8193       &6.9918E-03 	      &-1.2140
\\
\hline
4&          1.7817E-02 &1.8220 	   &2.9561E-02&1.9350         &	9.2100E-03        &-0.3975
\\
\hline
8&           3.8874E-03&2.1964        &6.0574E-03&	2.2869      &5.3950E-03          &0.7716
\\
\hline
16&         8.1581E-04 & 2.2525        & 1.2029E-03 &2.3321       &2.2502E-03	        &1.2616
\\
\hline
32&		      1.8214E-04&	2.1632      &2.5723E-04&	2.2254     &8.9122E-04         &1.3362
\\
\hline
\end{tabular}
\end{center}
\end{table}

\begin{table}[H]
\begin{center}
\caption{Numerical rates of convergence for the $C^{-1}-P_1(T)/P_1(\partial T)/P_0(T)$ element with the exact solution $\lambda=\cos(x)\cos(y)$ on the unit square domain $\Omega_1$; uniform triangular partitions; convection vector $\bbeta=[1, 1]$; reaction coefficient $c=1$; and the parameters $(\tau_1, \tau_2)=(0, 1)$.}\label{NE:TRI:Case1-1-0-1}
\begin{tabular}{|c|c|c|c|c|c|c|}
\hline
$1/h$        & $\|\epsilon_0\| $ &  order&  $\|\epsilon_b \| $ &  order&  $\|e_h \|$ &order
\\
\hline
1&        2.1739E-01  &&	         4.3538E-01 &                  & 8.0772E-03 	&
\\
\hline
2&        5.7342E-02 &	1.9227 &	1.0650E-01 & 2.0313       &1.4788E-02 	      &-0.8725
\\
\hline
4&          1.2879E-02 &2.1546 	   &2.2656E-02&	2.2329    &	 1.0705E-02        &0.4661
\\
\hline
8&            2.7182E-03&2.2443        &4.6597E-03&	2.2816    &5.9136E-03          &0.8562
\\
\hline
16&          5.9160E-04 & 2.1999        & 1.0064E-03 &2.2111    &3.0013E-03	        &0.9784
\\
\hline
32&		      1.3639E-04&	2.1169      &2.3136E-04&	2.1210    &1.5008E-03         &0.9999
\\
\hline
\end{tabular}
\end{center}
\end{table}

\begin{table}[H]
\begin{center}
\caption{Numerical rates of convergence for the $C^{-1}-P_1(T)/P_1(\partial T)/P_0(T)$ element with the exact solution $\lambda=\cos(x)\cos(y)$ on the unit square domain $\Omega_1$; uniform triangular partitions; convection vector $\bbeta=[1, 1]$; reaction coefficient $c=1$; and the parameters $(\tau_1, \tau_2)=(0, 0)$.}\label{NE:TRI:Case1-1-0-0}
\begin{tabular}{|c|c|c|c|c|c|c|}
\hline
$1/h$        & $\|\epsilon_0\| $ &  order&  $\|\epsilon_b \| $ &  order&  $\|e_h \|$ &order
\\
\hline
1&        1.9486E-01  &&	          3.9282E-01 &                  & 1.7572E-02 	&
\\
\hline
2&         4.6577E-02 &	2.0648 &	8.9249E-02 & 2.1380	       &1.7877E-02 	      &-0.0248
\\
\hline
4&           1.0883E-02 &2.0976 	   &1.9684E-02&	2.1808    &	 1.1270E-02        &0.6656
\\
\hline
8&            2.4728E-03&2.1379        &4.3116E-03&	2.1908    &5.9859E-03          &0.9128
\\
\hline
16&          5.6872E-04 & 2.1203        & 9.7480E-04 &2.1450    &3.0096E-03	        &0.9920
\\
\hline
32&		      1.3458E-04&	2.0793      &2.2889E-04&	2.0904     &1.5017E-03         &1.0030
\\
\hline
\end{tabular}
\end{center}
\end{table}

\begin{table}[H]
\begin{center}
\caption{Numerical rates of convergence for the $C^{-1}-P_1(T)/P_1(\partial T)/P_0(T)$ element with the exact solution $\lambda=\cos(x)\cos(y)$ on the unit square domain $\Omega_1$; uniform rectangular partitions; convection vector $\bbeta=[1, 1]$; reaction coefficient $c=1$; and the parameters $(\tau_1, \tau_2)=(1, 1)$.}\label{NE:REC:Case1-1-1-1}
\begin{tabular}{|c|c|c|c|c|c|c|}
\hline
$1/h$        & $\|\epsilon_0\| $ &  order&  $\|\epsilon_b \| $ &  order&  $\|e_h \|$ &order
\\
\hline
1&        5.7715E-02  &&	        1.7065E-01 &                  &3.3964E-02 	&
\\
\hline
2&         1.7618E-02 &	1.7119   &	4.3922E-02 &1.9580        &6.5677E-03 	      &2.3706
\\
\hline
4&         5.2857E-03 &1.7368  &1.1482E-02&1.9355             &1.8300E-03        &1.8435
\\
\hline
8&          1.3192E-03&2.0024        &2.5935E-03&2.1465        &8.3184E-04          &1.1375
\\
\hline
16&        3.1970E-04 &2.0449    &5.8821E-04 &2.1405         &3.2740E-04          &1.3453
\\
\hline
32&		7.8030E-05 &2.0346	    &1.3792E-04 &2.0926     	 &1.1616E-04       &1.4949
\\
\hline
\end{tabular}
\end{center}
\end{table}

\begin{table}[H]
\begin{center}
\caption{Numerical rates of convergence for the $C^{-1}-P_1(T)/P_1(\partial T)/P_0(T)$ element with the exact solution $\lambda=\cos(x)\cos(y)$ on the unit square domain $\Omega_1$; uniform rectangular partitions; convection vector $\bbeta=[1, 1]$; reaction coefficient $c=1$; and the parameters $(\tau_1, \tau_2)=(0, 1)$.}\label{NE:REC:Case1-1-0-1}
\begin{tabular}{|c|c|c|c|c|c|c|}
\hline
$1/h$        & $\|\epsilon_0\| $ &  order&  $\|\epsilon_b \| $ &  order&  $\|e_h \|$ &order
\\
\hline
1&        9.5831E-02  &&	        2.6881E-01 &                  &1.4294E-03 	&
\\
\hline
2&         2.0168E-02 &	2.2484   &	4.9829E-02 &2.4315        &2.2216E-03 	      &-0.6362
\\
\hline
4&         5.4250E-03 &1.8944   &1.1816E-02&2.0761              &1.8549E-03        &0.2603
\\
\hline
8&          1.3251E-03&2.0335        &2.6149E-03&2.1760        &8.9745E-04          &1.0475
\\
\hline
16&         3.1991E-04 &2.0504    &5.9016E-04 &2.1476         &3.3792E-04          &1.4092
\\
\hline
32&		 7.8039E-05 &2.0354	    &1.3815E-04 &2.0948     	 &1.1614E-04       &1.5408
\\
\hline
\end{tabular}
\end{center}
\end{table}

\begin{table}[H]
\begin{center}
\caption{Numerical rates of convergence for the $C^{-1}-P_1(T)/P_1(\partial T)/P_0(T)$ element with the exact solution $\lambda=\cos(x)\cos(y)$ on the unit square domain $\Omega_1$; uniform rectangular partitions; convection vector $\bbeta=[1, 1]$; reaction coefficient $c=1$; and the parameters $(\tau_1, \tau_2)=(0, 0)$.}\label{NE:REC:Case1-1-0-0}
\begin{tabular}{|c|c|c|c|c|c|c|}
\hline
$1/h$        & $\|\epsilon_0\| $ &  order&  $\|\epsilon_b \| $ &  order&  $\|e_h \|$ &order
\\
\hline
1&        9.5230E-02  &&	        2.6695E-01 &                  &1.5187E-03 	&
\\
\hline
2&         1.9874E-02 &	2.2605   &	4.9013E-02 &2.4453        &2.2637E-03 	      &-0.5758
\\
\hline
4&         5.3456E-03 &1.8945   &1.1620E-02&2.0765              &1.8751E-03        &0.2718
\\
\hline
8&          1.3145E-03&2.0238        &2.5890E-03&2.1661        &9.0089E-04          &1.0575
\\
\hline
16&         3.1893E-04 &2.0432    &5.8772E-04 &2.1392         &3.3841E-04          &1.4126
\\
\hline
32&		  7.7962E-05 &2.0324	    &1.3796E-04 &2.0909     	 &1.1621E-04       &1.5420
\\
\hline
\end{tabular}
\end{center}
\end{table}

Tables \ref{NE:TRI:Case1-L-1-1-1} - \ref{NE:TRI:Case1-L-1-0-0} present the numerical performance of $C^{-1}-P_1(T)/P_1(\partial T)/P_0(T)$ element for the uniform triangular partition on the non-convex L-shaped domain $\Omega_2$. The exact solution is $\lambda=\cos(x)\cos(y)$; the convection vector is $\bbeta=[1, 1]$; and the reaction coefficient is $c=1$. The parameters are taken as $(\tau_1, \tau_2)=(1, 1)$, $(\tau_1, \tau_2)=(0, 1)$, and $(\tau_1, \tau_2)=(0, 0)$, respectively. Tables \ref{NE:TRI:Case1-L-1-1-1} - \ref{NE:TRI:Case1-L-1-0-0} show that the convergence order for $\epsilon_0$ and $\epsilon_b$ in the discrete norm arrives at the optimal order of ${\cal O}(h^2)$.

\begin{table}[H]
\begin{center}
\caption{Numerical rates of convergence for the $C^{-1}-P_1(T)/P_1(\partial T)/P_0(T)$ element with the exact solution $\lambda=\cos(x)\cos(y)$ on the L-shaped domain $\Omega_2$; uniform triangular partitions; convection vector $\bbeta=[1, 1]$; reaction coefficient $c=1$; and the parameters $(\tau_1, \tau_2)=(1, 1)$.}\label{NE:TRI:Case1-L-1-1-1}
\begin{tabular}{|c|c|c|c|c|c|c|}
\hline
$1/h$        & $\|\epsilon_0\| $ &  order&  $\|\epsilon_b \| $ &  order&  $\|e_h \|$ &order
\\
\hline
1&        3.1337E-02  &&	        5.7371E-02 &                  &5.4743E-03 	&
\\
\hline
2&        8.5496E-03 &	1.8739 &	1.4445E-02 & 1.9897      &4.4327E-03 	      &0.3045
\\
\hline
4&         2.1599E-03 &1.9849	   &3.3992E-03&	2.0873    &	3.6959E-03        &0.2622
\\
\hline
8&         4.5758E-04&2.2389        &6.5979E-04&	2.3651    &1.8038E-03          &1.0349
\\
\hline
16&        1.0559E-04 &2.1155    &1.4455E-04 &2.1904    &6.9255E-04          &1.3810
\\
\hline
32&	2.5279E-05	   &2.0625	    &3.3466E-05 &2.1109	 &  2.9068E-04     &1.2525
\\
\hline
\end{tabular}
\end{center}
\end{table}

\begin{table}[H]
\begin{center}
\caption{Numerical rates of convergence for the $C^{-1}-P_1(T)/P_1(\partial T)/P_0(T)$ element with the exact solution $\lambda=\cos(x)\cos(y)$ on the L-shaped domain $\Omega_2$; uniform triangular partitions; convection vector $\bbeta=[1, 1]$; reaction coefficient $c=1$; and the parameters $(\tau_1, \tau_2)=(0, 1)$.}\label{NE:TRI:Case1-L-1-0-1}
\begin{tabular}{|c|c|c|c|c|c|c|}
\hline
$1/h$        & $\|\epsilon_0\| $ &  order&  $\|\epsilon_b \| $ &  order&  $\|e_h \|$ &order
\\
\hline
1&        2.9300E-02  &&	         5.9166E-02 &                  &1.2609E-02 	&
\\
\hline
2&         6.1821E-03 &	2.2448 &	1.2359E-02 & 2.2592       &8.2409E-03 	      &0.6136
\\
\hline
4&         1.3211E-03 &2.2264	   &2.6236E-03&	2.2359    &	4.6023E-03        &0.8404
\\
\hline
8&           2.6988E-04&2.2913        &5.5781E-04&	2.2337    &2.3407E-03          &0.9754
\\
\hline
16&         6.2113E-05 &2.1194     & 1.3109E-04 &2.0892    &1.1693E-03          &1.0013
\\
\hline
32&	1.4693E-05	   &2.0798    &3.1540E-05&2.0553	 &  5.8499E-04     &0.9991
\\
\hline
\end{tabular}
\end{center}
\end{table}

\begin{table}[H]
\begin{center}
\caption{Numerical rates of convergence for the $C^{-1}-P_1(T)/P_1(\partial T)/P_0(T)$ element with the exact solution $\lambda=\cos(x)\cos(y)$ on the L-shaped domain $\Omega_2$; uniform triangular partitions; convection vector $\bbeta=[1, 1]$; reaction coefficient $c=1$; and the parameters $(\tau_1, \tau_2)=(0, 0)$.}\label{NE:TRI:Case1-L-1-0-0}
\begin{tabular}{|c|c|c|c|c|c|c|}
\hline
$1/h$        & $\|\epsilon_0\| $ &  order&  $\|\epsilon_b \| $ &  order&  $\|e_h \|$ &order
\\
\hline
1&        2.8906E-02  &&	         5.8716E-02 &                  & 1.3434E-02 	&
\\
\hline
2&         6.1177E-03 &	2.2403 &	1.2295E-02 & 2.2557	       &8.3837E-03 	      &0.6803
\\
\hline
4&          1.3091E-03 &2.2244	   &2.6101E-03&	2.2359    &	4.6248E-03        &0.8582
\\
\hline
8&            2.6892E-04&2.2834        &5.5695E-04&	2.2285    &2.3432E-03          &0.9809
\\
\hline
16&         6.2049E-05 &2.1157     & 1.3104E-04 &2.0875    &1.1696E-03          &1.0025
\\
\hline
32&	1.4888E-05	   &2.0592	    &3.1537E-05&2.0549	 &5.8502E-04       &0.9994
\\
\hline
\end{tabular}
\end{center}
\end{table}

In Tables \ref{NE:TRI:Case1-2-1-1}-\ref{NE:TRI:Case1-L-2-0-0}, we demonstrate the numerical performance of the $C^{-1}-P_2(T)/P_2(\partial T)/P_1(T)$ element on uniform triangular partitions of the unit square domain $\Omega_1$ and the L-shaped domain $\Omega_2$, respectively. The exact solution is $\lambda=\cos(x)\cos(y)$; the convection vector is $\bbeta=[1, 1]$; and the reaction coefficient is $c=1$. The numerical results show that the convergence rates for $\epsilon_0$ and $\epsilon_b$ in the discrete $L^2$ norm are of the optimal order ${\cal O}(h^3)$ when the parameters are chosen as $(\tau_1, \tau_2)=(1, 1)$,  $(\tau_1, \tau_2)=(0, 1)$, and $(\tau_1, \tau_2)=(0, 0)$, respectively.

\begin{table}[H]
\begin{center}
\caption{Numerical rates of convergence for the $C^{-1}-P_2(T)/P_2(\partial T)/P_1(T)$ element with the exact solution $\lambda=\cos(x)\cos(y)$ on the unit square domain $\Omega_1$; uniform triangular partitions; convection vector $\bbeta=[1, 1]$; reaction coefficient $c=1$; and the parameters $(\tau_1, \tau_2)=(1, 1)$.}\label{NE:TRI:Case1-2-1-1}
\begin{tabular}{|c|c|c|c|c|c|c|}
\hline
$1/h$        & $\|\epsilon_0\| $ &  order&  $\|\epsilon_b \| $ &  order&  $\|e_h \|$ &order
\\
\hline
1&        2.2930E-02  &&	       3.7827E-02 &                  & 5.2346E-03 	&
\\
\hline
2&        2.7578E-03 &	3.0556      &4.7515E-03 &2.9929       &1.8394E-03 	      &1.5088
\\
\hline
4&          3.1406E-04 &3.1344 	   &5.5267E-04&3.1039        &	7.8572E-04        &1.2272
\\
\hline
8&           3.6798E-05&3.0933       &6.5190E-05&3.0837      &2.2526E-04          &1.8024
\\
\hline
16&         4.4211E-06 &3.0572       & 7.7955E-06 &3.0639       &5.9157E-05	        &1.9289
\\
\hline
32&		      5.4026E-07&3.0327    &9.4711E-07&	3.0410     &1.5124E-05         &1.9677
\\
\hline
\end{tabular}
\end{center}
\end{table}

\begin{table}[H]
\begin{center}
\caption{Numerical rates of convergence for the $C^{-1}-P_2(T)/P_2(\partial T)/P_1(T)$ element with the exact solution $\lambda=\cos(x)\cos(y)$ on the unit square domain $\Omega_1$; uniform triangular partitions; convection vector $\bbeta=[1, 1]$; reaction coefficient $c=1$; and the parameters $(\tau_1, \tau_2)=(0, 1)$.}\label{NE:TRI:Case1-2-0-1}
\begin{tabular}{|c|c|c|c|c|c|c|}
\hline
$1/h$        & $\|\epsilon_0\| $ &  order&  $\|\epsilon_b \| $ &  order&  $\|e_h \|$ &order
\\
\hline
1&        1.0238E-02  &&	        1.9301E-02 &                  &3.8199E-04 	&
\\
\hline
2&       2.4224E-03 &	2.0794   &	4.2489E-03 & 2.1835       &2.3881E-03 	      &-2.6443
\\
\hline
4&         2.8603E-04 &3.0822	   &5.4886E-04&	2.9526   &	 1.0196E-03        &1.2278
\\
\hline
8&           3.1690E-05&3.1741       &6.4620E-05&	3.0864   &3.0835E-04          &1.7254
\\
\hline
16&         3.5243E-06 & 3.1686       &7.5240E-06 &3.1024    &8.3450E-05	        &1.8856
\\
\hline
32&		    4.0429E-07&3.1239     &8.9192E-07&	3.0765   &2.1622E-05         &1.9484
\\
\hline
\end{tabular}
\end{center}
\end{table}

\begin{table}[H]
\begin{center}
\caption{Numerical rates of convergence for the $C^{-1}-P_2(T)/P_2(\partial T)/P_1(T)$ element with the exact solution $\lambda=\cos(x)\cos(y)$ on the unit square domain $\Omega_1$; uniform triangular partitions; convection vector $\bbeta=[1, 1]$; reaction coefficient $c=1$; and the parameters $(\tau_1, \tau_2)=(0, 0)$.}\label{NE:TRI:Case1-2-0-0}
\begin{tabular}{|c|c|c|c|c|c|c|}
\hline
$1/h$        & $\|\epsilon_0\| $ &  order&  $\|\epsilon_b \| $ &  order&  $\|e_h \|$ &order
\\
\hline
1&        1.0603E-02 &&	         1.9936E-02 &                  & 6.6104E-04 	&
\\
\hline
2&         2.8172E-03 &1.9122&    	4.9908E-03 &1.9980      &3.1747E-03 	      &-2.2638
\\
\hline
4&           2.9222E-04 &3.2691 	   &5.6321E-04&	3.1475   &	1.1058E-03        &1.5216
\\
\hline
8&            3.1924E-05&3.1944       &6.5173E-05&	3.1113    &3.1349E-04          &1.8186
\\
\hline
16&          3.5320E-06 & 3.1761        &7.5429E-06 &3.1111   &8.3750E-05	        &1.9042
\\
\hline
32&		     4.0453E-07&3.1262     &8.9255E-07&	3.0791     &2.1640E-05         &1.9524
\\
\hline
\end{tabular}
\end{center}
\end{table}

\begin{table}[H]
\begin{center}
\caption{Numerical rates of convergence for the $C^{-1}-P_2(T)/P_2(\partial T)/P_1(T)$ element with the exact solution $\lambda=\cos(x)\cos(y)$ on the L-shaped domain $\Omega_2$; uniform triangular partitions; convection vector $\bbeta=[1, 1]$; reaction coefficient $c=1$; and the parameters $(\tau_1, \tau_2)=(1, 1)$.}\label{NE:TRI:Case1-L-2-1-1}
\begin{tabular}{|c|c|c|c|c|c|c|}
\hline
$1/h$        & $\|\epsilon_0\| $ &  order&  $\|\epsilon_b \| $ &  order&  $\|e_h \|$ &order
\\
\hline
1&        1.8602E-03  &&	        3.1654E-03 &                  &2.9270E-03 	&
\\
\hline
2&         2.8228E-04 &	2.7202   &	5.0711E-04 &2.6420         &8.2733E-04 	      &1.8229
\\
\hline
4&          3.4014E-05 &3.0529   &6.1046E-05&3.0543              &	2.2506E-04        &1.8781
\\
\hline
8&           4.1542E-06&3.0335        &7.4088E-06&	  3.0426        &5.7479E-05          &1.9692
\\
\hline
16&         5.1021E-07 &3.0254    &9.0482E-07 &3.0335          &1.4583E-05          &1.9787
\\
\hline
32&		6.3109E-08   &	3.0152    &1.1151E-07&3.0204	 & 3.6764E-06      &1.9879
\\
\hline
\end{tabular}
\end{center}
\end{table}

\begin{table}[H]
\begin{center}
\caption{Numerical rates of convergence for the $C^{-1}-P_2(T)/P_2(\partial T)/P_1(T)$ element with the exact solution $\lambda=\cos(x)\cos(y)$ on the L-shaped domain $\Omega_2$; uniform triangular partitions; convection vector $\bbeta=[1, 1]$; reaction coefficient $c=1$; and the parameters $(\tau_1,\tau_2)=(0, 1)$.}\label{NE:TRI:Case1-L-2-0-1}
\begin{tabular}{|c|c|c|c|c|c|c|}
\hline
$1/h$        & $\|\epsilon_0\| $ &  order&  $\|\epsilon_b \| $ &  order&  $\|e_h \|$ &order
\\
\hline
1&        2.3256E-03  &&	        4.3470E-03 &                  &2.6712E-03 	&
\\
\hline
2&         2.7541E-04 &	3.0780   &	5.3573E-04 &3.0204            &1.0506E-03 	      &1.3463
\\
\hline
4&          2.9682E-05 &3.2140   &6.1448E-05&3.1241              &	3.0201E-04        &1.7985
\\
\hline
8&           3.3257E-06&3.1579        &7.1909E-06&	  3.0951         &8.0495E-05          &1.9076
\\
\hline
16&         3.8296E-07 &3.1184     &8.5521E-07 &3.0718          &2.0782E-05          &1.9536
\\
\hline
32&		4.5478E-08   &	3.0740    &1.0369E-07 &	3.0441 & 5.2789E-06      &1.9770
\\
\hline
\end{tabular}
\end{center}
\end{table}

\begin{table}[H]
\begin{center}
\caption{Numerical rates of convergence for the $C^{-1}-P_2(T)/P_2(\partial T)/P_1(T)$ element with the exact solution $\lambda=\cos(x)\cos(y)$ on the L-shaped domain $\Omega_2$; uniform triangular partitions; convection vector $\bbeta=[1, 1]$; reaction coefficient $c=1$; and the parameters $(\tau_1, \tau_2)=(0, 0)$.}\label{NE:TRI:Case1-L-2-0-0}
\begin{tabular}{|c|c|c|c|c|c|c|}
\hline
$1/h$        & $\|\epsilon_0\| $ &  order&  $\|\epsilon_b \| $ &  order&  $\|e_h \|$ &order
\\
\hline
1&        2.4722E-03  &&	       4.6393E-03 &                  &2.8720E-03 	&
\\
\hline
2&         2.7900E-04 &	3.1474   &	5.4309E-04 &3.0946            &1.0730E-03 	      &1.4204
\\
\hline
4&         2.9790E-05 &3.2274   &6.1690E-05&3.1381              &	3.0360E-04        &1.8214
\\
\hline
8&         3.3289E-06&3.1617        &7.1986E-06&3.0993         &8.0598E-05          &1.9134
\\
\hline
16&         3.8306E-07 &3.1194     &8.5546E-07 &3.0730         &2.0789E-05          &1.9549
\\
\hline
32&	4.5481E-08	   &3.0742	    &1.0369E-07&3.0444	 &5.2794E-04      &1.9774
\\
\hline
\end{tabular}
\end{center}
\end{table}

\subsection {Continuous convection vector $\bbeta$}\label{case2:variable coeff}  The numerical experiments in this subsection assume the exact solution $\lambda=\exp(x)\cos(y)$, the convection vector is given by $\bbeta=[0.5-y, x-0.5]$; the reaction coefficient is $c=0$, the parameters are taken as $(\tau_1, \tau_2)=(1, 1)$, \ $(\tau_1, \tau_2)=(0, 1)$ and $(\tau_1, \tau_2)=(0, 0)$.

Tables \ref{NE:TRI:£ºL-Case2-1-1-1}-\ref{NE:TRI:£ºCrack-Case2-1-0-0} show the numerical performance of the $C^{-1}-P_1(T)/P_1(\partial T)/P_0(T)$ element on the uniform triangular partition of the  L-shaped domain $\Omega_2$ and the cracked unit square domain $\Omega_3$, respectively.  The numerical results in Tables \ref{NE:TRI:£ºL-Case2-1-1-1}-\ref{NE:TRI:£ºCrack-Case2-1-0-0} indicate that the convergence rates for $\epsilon_0$ and $\epsilon_b$ in the discrete $L^2$ norm arrive at the optimal order of ${\cal O}(h^{2})$, which are consistent with the theory.

\begin{table}[H]
\begin{center}
\caption{Numerical rates of convergence for the $C^{-1}-P_1(T)/P_1(\partial T)/P_0(T)$ element with the exact solution $\lambda=\exp(x)\cos(y)$ on the L-shaped domain $\Omega_2$; uniform triangular partitions; convection vector $\bbeta=[0.5-y, x-0.5]$; reaction term $c=0$; and the parameters $(\tau_1, \tau_2)=(1, 1)$.}\label{NE:TRI:£ºL-Case2-1-1-1}
\begin{tabular}{|c|c|c|c|c|c|c|}
\hline
$1/h$        & $\|\epsilon_0\| $ &  order&  $\|\epsilon_b \| $ &  order&  $\|e_h \|$ &order
\\
\hline
1&       3.2415E-02  &&	      6.3078E-02 &                  &1.6843E-02 	&
\\
\hline
2&        1.1680E-02 &1.4726    &2.0531E-02 &1.6193             &2.2578E-02 	      &-0.4228
\\
\hline
4&         3.1123E-03 &1.9080	   &5.3381E-03&1.9434           &	2.2642E-02        &-0.0041
\\
\hline
8&        7.3609E-04&2.0800    &1.2796E-03&2.0607              & 1.7348E-02          &0.3843
\\
\hline
16&      1.7703E-04& 2.0559      &3.1225E-04&2.0349            &	1.1446E-02        &0.6000
\\
\hline
32&		4.4053E-05    & 2.0067  &7.7975E-05  &	2.0016     &  6.9526E-03     &0.7192
\\
\hline
\end{tabular}
\end{center}
\end{table}

\begin{table}[H]
\begin{center}
\caption{Numerical rates of convergence for the $C^{-1}-P_1(T)/P_1(\partial T)/P_0(T)$ element with the exact solution $\lambda=\exp(x)\cos(y)$ on the L-shaped domain $\Omega_2$; uniform triangular partitions; convection vector $\bbeta=[0.5-y, x-0.5]$; reaction term $c=0$; and the parameters $(\tau_1, \tau_2)=(0, 1)$.}\label{NE:TRI:£ºL-Case2-1-0-1}
\begin{tabular}{|c|c|c|c|c|c|c|}
\hline
$1/h$        & $\|\epsilon_0\| $ &  order&  $\|\epsilon_b \| $ &  order&  $\|e_h \|$ &order
\\
\hline
1&       1.9423E-02  &&	      3.8931E-02 &                  &1.8311E-02 	&
\\
\hline
2&        1.0108E-02 &0.9422    &1.7889E-02 &1.1219             &2.4481E-02 	      &-0.4190
\\
\hline
4&         3.1700E-03 &1.6730	   &5.4635E-03&1.7112           &	2.3976E-02        &0.0300
\\
\hline
8&         7.3660E-04&2.1055    &1.2945E-03&2.0774              & 1.8087E-02          &0.4067
\\
\hline
16&      1.7486E-04& 2.0746      &3.1319E-04&2.0473            &	1.1790E-02        &0.6173
\\
\hline
32&	4.3373E-05	    &2.0114   &7.8059E-05  &2.0044	     &  7.1064E-03    &0.7304
\\
\hline
\end{tabular}
\end{center}
\end{table}

\begin{table}[H]
\begin{center}
\caption{Numerical rates of convergence for the $C^{-1}-P_1(T)/P_1(\partial T)/P_0(T)$ element with the exact solution $\lambda=\exp(x)\cos(y)$ on the L-shaped domain $\Omega_2$; uniform triangular partitions; convection vector $\bbeta=[0.5-y, x-0.5]$; reaction $c=0$; and the parameters $(\tau_1, \tau_2)=(0, 0)$.}\label{NE:TRI:£ºL-Case2-1-0-0}
\begin{tabular}{|c|c|c|c|c|c|c|}
\hline
$1/h$        & $\|\epsilon_0\| $ &  order&  $\|\epsilon_b \| $ &  order&  $\|e_h \|$ &order
\\
\hline
1&       6.1648E-02  &&	      1.2097E-01 &                  &3.9537E-01 	&
\\
\hline
2&        1.6157E-02 &1.9319     &2.9034E-02 &2.0588             &1.6059E-01 	      &1.2999
\\
\hline
4&         2.9637E-03 &2.4468	   &5.4408E-03&2.4159           &	9.2294E-02        &0.7991
\\
\hline
8&         7.1138E-04&2.0587    &1.2945E-03&2.0714              & 4.8769E-02          &0.9203
\\
\hline
16&      1.7444E-04& 2.0279       &3.1599E-04&2.0345            &	2.5783E-02        &0.9196
\\
\hline
32&		4.3420E-05    & 2.0063  &7.8403E-05 &	2.0109     & 1.3545E-02      &0.9287
\\
\hline
\end{tabular}
\end{center}
\end{table}

\begin{table}[H]
\begin{center}
\caption{Numerical rates of convergence for the $C^{-1}-P_1(T)/P_1(\partial T)/P_0(T)$ element with the exact solution $\lambda=\exp(x)\cos(y)$ on the cracked domain $\Omega_3$; uniform triangular partitions; convection $\bbeta=[0.5-y, x-0.5]$; reaction $c=0$; and the parameters $(\tau_1, \tau_2)=(1, 1)$.}\label{NE:TRI:£ºCrack-Case2-1-1-1}
\begin{tabular}{|c|c|c|c|c|c|c|}
\hline
$1/h$        & $\|\epsilon_0\| $ &  order&  $\|\epsilon_b \| $ &  order&  $\|e_h \|$ &order
\\
\hline
1&        3.3510E-02  &   &	      6.4966E-02 &                  &2.1932E-02 	&
\\
\hline
2&       1.2997E-02 &1.3665     &2.2720E-02 &1.5157      &2.9850E-02 	      &-0.4447
\\
\hline
4&        4.8193E-03 &1.4312 	   &8.0114E-03&1.5039     &3.0696E-02        &-0.0403
\\
\hline
8&        1.6439E-03&1.5517     &2.6809E-03&1.5794     &2.3851E-02          &0.3640
\\
\hline
16&      4.3889E-04&1.9052        &7.1140E-04 &1.9140      &1.5641E-02	     &0.6087
\\
\hline
32&	1.1360E-04	    &1.9498   & 1.8362E-04     &	1.9539     &  9.4027E-03     &0.7342
\\
\hline
\end{tabular}
\end{center}
\end{table}

\begin{table}[H]
\begin{center}
\caption{Numerical rates of convergence for the $C^{-1}-P_1(T)/P_1(\partial T)/P_0(T)$ element with the exact solution $\lambda=\exp(x)\cos(y)$ on the cracked domain $\Omega_3$; uniform triangular partitions; convection $\bbeta=[0.5-y, x-0.5]$; reaction $c=0$; and the parameters $(\tau_1, \tau_2)=(0, 1)$.}\label{NE:TRI:£ºCrack-Case2-1-0-1}
\begin{tabular}{|c|c|c|c|c|c|c|}
\hline
$1/h$        & $\|\epsilon_0\| $ &  order&  $\|\epsilon_b \| $ &  order&  $\|e_h \|$ &order
\\
\hline
1&        1.9836E-02  &   &	      3.9620E-02 &                  &2.3612E-02 	&
\\
\hline
2&        1.1357E-02 &0.8046     &2.0050E-02 &0.9827      &3.2625E-02 	      &-0.4664
\\
\hline
4&        5.2721E-03 &1.1071 	   &8.7776E-03&1.1917     &3.2435E-02        &0.0084
\\
\hline
8&        1.7029E-03&1.6303     &2.7865E-03&1.6554      &2.4813E-02          &0.3865
\\
\hline
16&      4.4221E-04&1.9452         &7.1948E-04 &1.9534      &1.6095E-02	     &0.6245
\\
\hline
32&	1.1363E-04	    & 1.9604  &  1.8436E-04    &	1.9645     &  9.6079E-03     &0.7444
\\
\hline
\end{tabular}
\end{center}
\end{table}

\begin{table}[H]
\begin{center}
\caption{Numerical rates of convergence for the $C^{-1}-P_1(T)/P_1(\partial T)/P_0(T)$ element with the exact solution $\lambda=\exp(x)\cos(y)$ on the cracked domain $\Omega_3$; uniform triangular partitions; convection $\bbeta=[0.5-y, x-0.5]$; reaction $c=0$; and the parameters $(\tau_1, \tau_2)=(0, 0)$.}\label{NE:TRI:£ºCrack-Case2-1-0-0}
\begin{tabular}{|c|c|c|c|c|c|c|}
\hline
$1/h$        & $\|\epsilon_0\| $ &  order&  $\|\epsilon_b \| $ &  order&  $\|e_h \|$ &order
\\
\hline
1&        1.5431E-01  &   &	      2.7560E-01 &                  &8.0755E-01 	&
\\
\hline
2&        3.4877E-02 &2.1455     &5.9402E-02 &2.2140      &2.5660E-01 	      &1.6540
\\
\hline
4&        7.0917E-03 &2.2981 	   &1.1811E-02&2.3304      &1.4455E-01        &0.8280
\\
\hline
8&         1.8011E-03&1.9772      &2.9530E-03&1.9999       &7.4987E-02          &0.9469
\\
\hline
16&      4.5484E-04&1.9855          &7.4046E-04 &1.9957       &3.9028E-02	     &0.9421
\\
\hline
32&	1.1483E-04	    & 1.9858  &1.8634E-04      &1.9905	     &2.0217E-02       &0.9490
\\
\hline
\end{tabular}
\end{center}
\end{table}

Tables \ref{NE:TRI:£ºL-Case2-2-1-1}-\ref{NE:TRI:£ºCrack-Case2-2-0-0} illustrate the numerical results of the $C^{-1}-P_2(T)/P_2(\partial T)/P_1(T)$ element on the uniform triangular partition of the L-shaped domain $\Omega_2$ and the cracked unit square domain $\Omega_3$, respectively. The exact solution is given as $\lambda=\exp(x)\cos(y)$; the convection is $\bbeta=(0.5-y, x-0.5)$; and the reaction is $c=0$. The parameters are given by $(\tau_1, \tau_2)=(1, 1)$, $(\tau_1, \tau_2)=(0, 1)$ and $(\tau_1, \tau_2)=(0, 0)$, respectively. The numerical results in Tables \ref{NE:TRI:£ºL-Case2-2-1-1}-\ref{NE:TRI:£ºCrack-Case2-2-0-0} demonstrate that the convergence rates for $\epsilon_0$ and $\epsilon_b$ in the discrete $L^2$ norm arrive at the optimal order of ${\cal O}(h^3)$, which are consistent greatly with the theory.

\begin{table}[H]
\begin{center}
\caption{Numerical rates of convergence for the $C^{-1}-P_2(T)/P_2(\partial T)/P_1(T)$ element with the exact solution $\lambda=\exp(x)\cos(y)$ on the L-shaped domain $\Omega_2$; uniform triangular partitions; convection $\bbeta=[0.5-y, x-0.5]$; reaction $c=0$; and the parameters $(\tau_1, \tau_2)=(1, 1)$.}\label{NE:TRI:£ºL-Case2-2-1-1}
\begin{tabular}{|c|c|c|c|c|c|c|}
\hline
$1/h$        & $\|\epsilon_0\| $ &  order&  $\|\epsilon_b \| $ &  order&  $\|e_h \|$ &order
\\
\hline
1&      2.1302E-02  &&	      2.6687E-02 &                  &1.4892E-02 	&
\\
\hline
2&      2.3777E-03 &3.1634    &3.4624E-03 &2.9463           &6.4721E-03 	      &1.2023
\\
\hline
4&        1.3207E-04 &4.1702	   &1.9602E-04&4.1427           &1.9829E-03        &1.7066
\\
\hline
8&        9.2125E-06&3.8415  &1.4695E-05&3.7376             &5.9265E-04          &1.7424
\\
\hline
16&      9.5606E-07& 3.2684     &1.5840E-06&3.2137          &1.7710E-04        &1.7427
\\
\hline
32&	1.1556E-07	    & 3.0484  &1.9136E-07 &	3.0493     &5.1634E-05       &1.7781
\\
\hline
\end{tabular}
\end{center}
\end{table}

\begin{table}[H]
\begin{center}
\caption{Numerical rates of convergence for the $C^{-1}-P_2(T)/P_2(\partial T)/P_1(T)$ element  with the exact solution $\lambda=\exp(x)\cos(y)$ on the L-shaped domain $\Omega_2$; uniform triangular partitions; convection $\bbeta=[0.5-y, x-0.5]$; reaction $c=0$; and the parameters $(\tau_1,\tau_2)=(0, 1)$.}\label{NE:TRI:£ºL-Case2-2-0-1}
\begin{tabular}{|c|c|c|c|c|c|c|}
\hline
$1/h$        & $\|\epsilon_0\| $ &  order&  $\|\epsilon_b \| $ &  order&  $\|e_h \|$ &order
\\
\hline
1&      6.1064E-03  &&	      8.2459E-03 &                  &3.7357E-03 	&
\\
\hline
2&       1.2992E-03 &2.2327    &1.9879E-03 &2.0524            &3.2926E-03 	      &0.1822
\\
\hline
4&        9.3319E-05 &3.7993	   &1.4863E-04&3.7415           &1.6764E-03        &0.9739
\\
\hline
8&        8.4882E-06&3.4586   &1.4292E-05&3.3784              &5.9142E-04          &1.5031
\\
\hline
16&      9.5130E-07& 3.1575     &1.6239E-06&3.1377           &1.8308E-04        &1.6917
\\
\hline
32&	1.1425E-07	    &3.0578   &1.9409E-07   &	3.0647     &5.3445E-05       &1.7763
\\
\hline
\end{tabular}
\end{center}
\end{table}

\begin{table}[H]
\begin{center}
\caption{Numerical rates of convergence for the $C^{-1}-P_2(T)/P_2(\partial T)/P_1(T)$ element with the exact solution $\lambda=\exp(x)\cos(y)$ on the L-shaped domain $\Omega_2$; uniform triangular partitions; convection $\bbeta=[0.5-y, x-0.5]$; reaction $c=0$; and the parameters $(\tau_1, \tau_2)=(0, 0)$.}\label{NE:TRI:£ºL-Case2-2-0-0}
\begin{tabular}{|c|c|c|c|c|c|c|}
\hline
$1/h$        & $\|\epsilon_0\| $ &  order&  $\|\epsilon_b \| $ &  order&  $\|e_h \|$ &order
\\
\hline
1&       4.2050E-03  &&	      6.9789E-03 &                  &7.6006E-02 	&
\\
\hline
2&       8.6645E-04 &2.2789    &1.4333E-03 &2.2836             &1.7364E-02 	      &2.1300
\\
\hline
4&        7.7037E-05 &3.4915	   &1.3057E-04&3.4565           &	4.0621E-03        &2.0958
\\
\hline
8&        8.1542E-06&3.2399   &1.4018E-05&3.2195              &1.1028E-03          &1.8811
\\
\hline
16&      9.4470E-07& 3.1096      &1.6197E-06&3.1135           &2.9801E-04        &1.8877
\\
\hline
32&	1.1404E-07	    & 3.0503  &1.9391E-07 &	3.0622     & 7.9925E-05      &1.8986
\\
\hline
\end{tabular}
\end{center}
\end{table}

\begin{table}[H]
\begin{center}
\caption{Numerical rates of convergence for the $C^{-1}-P_2(T)/P_2(\partial T)/P_1(T)$ element with the exact solution $\lambda=\exp(x)\cos(y)$ on the cracked domain $\Omega_3$; uniform triangular partitions; convection $\bbeta=[0.5-y, x-0.5]$; reaction $c=0$; and the parameters $(\tau_1, \tau_2)=(1, 1)$.}\label{NE:TRI:£ºCrack-Case2-2-1-1}
\begin{tabular}{|c|c|c|c|c|c|c|}
\hline
$1/h$        & $\|\epsilon_0\| $ &  order&  $\|\epsilon_b \| $ &  order&  $\|e_h \|$ &order
\\
\hline
1&        2.6589E-02  &&	       3.2397E-02 &                  &1.8642E-02 	&
\\
\hline
2&        2.5601E-03 &3.3766      &3.6330E-03 &3.1566              &8.1520E-03 	      &1.1933
\\
\hline
4&         1.4880E-04 &4.1047	   &2.1653E-04&4.0685           &	2.5931E-03        &1.6525
\\
\hline
8&         1.0823E-05&3.7813     &1.6479E-05&3.7159               & 7.8293E-04          &1.7277
\\
\hline
16&      1.1555E-06& 3.2275       &1.7821E-06&3.2090            &	2.3144E-04        &1.7582
\\
\hline
32&	1.4005E-07	    & 3.0446  & 2.1456E-07  &	3.0542     & 6.6289E-05      &1.8038
\\
\hline
\end{tabular}
\end{center}
\end{table}

\begin{table}[H]
\begin{center}
\caption{Numerical rates of convergence for the $C^{-1}-P_2(T)/P_2(\partial T)/P_1(T)$ element with the exact solution $\lambda=\exp(x)\cos(y)$ on the cracked domain $\Omega_3$; uniform triangular partitions; convection $\bbeta=[0.5-y, x-0.5]$; reaction $c=0$; and the parameters $(\tau_1,\tau_2)=(0, 1)$.}\label{NE:TRI:£ºCrack-Case2-2-0-1}
\begin{tabular}{|c|c|c|c|c|c|c|}
\hline
$1/h$        & $\|\epsilon_0\| $ &  order&  $\|\epsilon_b \| $ &  order&  $\|e_h \|$ &order
\\
\hline
1&        6.4178E-03  &&	       8.7544E-03 &                  &5.9985E-03 	&
\\
\hline
2&        1.3561E-03 &2.2426     &2.0718E-03 &2.0791              &4.8223E-03 	      &0.3149
\\
\hline
4&         1.0909E-04 &3.6359	   &1.6899E-04&3.6159            &	2.2903E-03        &1.0742
\\
\hline
8&         9.9618E-06&3.4530     &1.5845E-05&3.4148               & 7.9144E-04          &1.5330
\\
\hline
16&      1.1336E-06& 3.1355        & 1.7974E-06&3.1401             &	2.4064E-04        &1.7176
\\
\hline
32&	1.3714E-07	    &3.0471   &2.1517E-07  &	3.0624     & 6.8983E-05      &1.8026
\\
\hline
\end{tabular}
\end{center}
\end{table}

\begin{table}[H]
\begin{center}
\caption{Numerical rates of convergence for the $C^{-1}-P_2(T)/P_2(\partial T)/P_1(T)$ element with the exact solution $\lambda=\exp(x)\cos(y)$ on the cracked domain $\Omega_3$; uniform triangular partitions; convection $\bbeta=[0.5-y, x-0.5]$; reaction $c=0$; and the parameters $(\tau_1, \tau_2)=(0, 0)$.}\label{NE:TRI:£ºCrack-Case2-2-0-0}
\begin{tabular}{|c|c|c|c|c|c|c|}
\hline
$1/h$        & $\|\epsilon_0\| $ &  order&  $\|\epsilon_b \| $ &  order&  $\|e_h \|$ &order
\\
\hline
1&        5.0789E-03  &&	       8.3811E-03 &                  &8.2058E-02 	&
\\
\hline
2&        9.6469E-04 &2.3964     &1.5700E-03 &2.4163      &2.0161E-02 	      &2.0251
\\
\hline
4&         9.1034E-05 &3.4056 	   &1.4755E-04&3.4116       &	4.9032E-03        &2.0398
\\
\hline
8&         9.7249E-06&3.2267      &1.5660E-05&3.2360     &1.3389E-03          &1.8727
\\
\hline
16&      1.1326E-06& 3.1021        & 1.7997E-06&3.1212       &	3.6043E-04        &1.8933
\\
\hline
32&	1.3717E-07	    &3.0456   &2.1528E-07  &3.0635	     &  9.6096E-05     &1.9072
\\
\hline
\end{tabular}
\end{center}
\end{table}

Tables \ref{NEE:TRI:Case-P1-0-1}-\ref{NEE:TRI:Case-P2-1-0}
illustrate the numerical results on the uniform triangular partition of the unit square domain $\Omega_1$. The exact solution is $\lambda=\sin(\pi x)\cos(\pi y)$; the convection is $\bbeta=[-y, x]$; and the reaction is $c=x+y$. The parameters are given by $(\tau_1, \tau_2)=(0, 1)$, $(\tau_1, \tau_2)=(1, 1)$, $(\tau_1, \tau_2)=(0, 0)$ and $(\tau_1, \tau_2)=(1, 0)$, respectively. The numerical results in Tables \ref{NEE:TRI:Case-P1-0-1}-\ref{NEE:TRI:Case-P1-1-0} demonstrate that the convergence rates for $\epsilon_0$ and $\epsilon_b$ in the discrete $L^2$ norm arrive at the optimal order of ${\cal O}(h^2)$ when the $C^{-1}-P_1(T)/P_1(\partial T)/P_0(T)$ element is employed. Tables \ref{NEE:TRI:Case-P2-0-1}-\ref{NEE:TRI:Case-P2-1-0} show that the convergence rates for $\epsilon_0$ and $\epsilon_b$ in the discrete $L^2$ norm are of the optimal order ${\cal O}(h^3)$ when the $C^{-1}-P_2(T)/P_2(\partial T)/P_1(T)$ element is used.

\begin{table}[H]
\begin{center}
\caption{Numerical rates of convergence for the $C^{-1}-P_1(T)/P_1(\partial T)/P_0(T)$ element  with the exact solution $\lambda=\sin(\pi x)\cos(\pi y)$ on the unit square domain $\Omega_1$; uniform triangular partitions; convection $\bbeta=[-y, x]$; reaction $c=x+y$; and the parameters $(\tau_1, \tau_2)=(0, 1)$.}\label{NEE:TRI:Case-P1-0-1}
\begin{tabular}{|c|c|c|c|c|c|c|}
\hline
$1/h$        & $\|\epsilon_0\| $ &  order&  $\|\epsilon_b \| $ &  order&  $\|e_h \|$&order
\\
\hline
1&        2.1635E-00  &&	       4.3282E-00 &                  &2.9913E-02 	&
\\
\hline
2&        6.8925E-01 &	1.6503 &	1.2858E-00 &	1.7511       &2.1249E-01	&  -2.8286
\\
\hline
4&         1.7782E-01 &1.9546	   &3.1162E-01&2.0448  &	1.4365E-01  &0.5648
\\
\hline
8&         3.4347E-02  &2.3722        &5.9643E-02& 2.3853    &9.7388E-02   &0.5608
\\
\hline
16&      7.8861E-03 & 2.1228      &1.3698E-02 &2.1224  &5.5483E-02	 &0.8117
\\
\hline
32&		1.9449E-03& 2.0196     &3.3694E-03& 2.0234   &2.9496E-02&0.9115
\\
\hline
\end{tabular}
\end{center}
\end{table}

\begin{table}[H]
\begin{center}
\caption{Numerical rates of convergence for the $C^{-1}-P_1(T)/P_1(\partial T)/P_0(T)$ element  with the exact solution $\lambda=\sin(\pi x)\cos(\pi y)$ on the unit square domain $\Omega_1$; uniform triangular partitions; convection $\bbeta=[-y, x]$; reaction $c=x+y$; and the parameters $(\tau_1, \tau_2)=(1, 1)$.}\label{NEE:TRI:Case-P1-1-1}
\begin{tabular}{|c|c|c|c|c|c|c|}
\hline
$1/h$        & $\|\epsilon_0\| $ &  order&  $\|\epsilon_b \| $ &  order&  $\|e_h \|$ &order
\\
\hline
1&      2.4003E-00  &&	        4.7903E-00 &                  &4.9317E-02 	&
\\
\hline
2&     6.8303E-01 &	1.8132 &	1.2888E-00 &1.8941       &1.7957E-01 	      &-1.8644
\\
\hline
4&       1.6889E-01 &2.0159	   &3.0049E-01&	2.1001    &1.4889E-01        &0.2702
\\
\hline
8&        3.2831E-02&2.3630     &5.6977E-02&2.3989   &9.5989E-02          &0.6333
\\
\hline
16&       7.7039E-03 & 2.0914       &1.3306E-02 &2.0983  &5.1431E-02	        &0.9002
\\
\hline
32&		 1.9429E-03& 1.9874     &3.3402E-03&1.9941   &2.6739E-02         &0.9437
\\
\hline
\end{tabular}
\end{center}
\end{table}

\begin{table}[H]
\begin{center}
\caption{Numerical rates of convergence for the $C^{-1}-P_1(T)/P_1(\partial T)/P_0(T)$ element  with the exact solution $\lambda=\sin(\pi x)\cos(\pi y)$ on the unit square domain $\Omega_1$; uniform triangular partitions; convection $\bbeta=[-y, x]$; reaction $c=x+y$; and the parameters $(\tau_1, \tau_2)=(0, 0)$.}\label{NEE:TRI:Case-P1-0-0}
\begin{tabular}{|c|c|c|c|c|c|c|}
\hline
$1/h$        & $\|\epsilon_0\| $ &  order&  $\|\epsilon_b \| $ &  order&  $\|e_h \|$ &order
\\
\hline
1&      3.1677E-00  &&	     6.4055E-00 &                  &1.2305E-00 	&
\\
\hline
2&      6.2169E-01 & 2.3491  &1.176E-00 &2.4451         &5.0819E-01 	      &1.2758
\\
\hline
4&        1.3100E-01 &2.2466	   &2.3434E-01&	2.3275    &2.6253E-01        &0.9529
\\
\hline
8&       3.1142E-02&2.0726    &5.4837E-02&2.0954   &1.2631E-01          &1.0554
\\
\hline
16&       7.7204E-03 & 2.0121       &1.3469E-02 &2.0255  &6.1632E-02	        &1.0353
\\
\hline
32&		1.9358E-03& 1.9957   &3.3574E-03&	2.0042  &3.0642E-02         &1.0082
\\
\hline
\end{tabular}
\end{center}
\end{table}

\begin{table}[H]
\begin{center}
\caption{Numerical rates of convergence for the $C^{-1}-P_1(T)/P_1(\partial T)/P_0(T)$ element  with the exact solution $\lambda=\sin(\pi x)\cos(\pi y)$ on the unit square domain $\Omega_1$; uniform triangular partitions; convection $\bbeta=[-y, x]$; reaction $c=x+y$; and the parameters $(\tau_1, \tau_2)=(1, 0)$.}\label{NEE:TRI:Case-P1-1-0}
\begin{tabular}{|c|c|c|c|c|c|c|}
\hline
$1/h$        & $\|\epsilon_0\| $ &  order&  $\|\epsilon_b \| $ &  order&  $\|e_h \|$ &order
\\
\hline
1&       2.7177E-00  &&	      5.5313E-00 &                  &1.3068E-00 	&
\\
\hline
2&       6.8277E-01 &1.9929    &	1.2846E-00 &2.1064      &6.1251E-01 	      &1.0932
\\
\hline
4&       1.3531E-01 &2.3351 	   &2.4105E-01&	2.4139    &2.9282E-01       &1.0647
\\
\hline
8&       3.1578E-02&2.0993      &5.5178E-02&2.1272   &1.2668E-01          &1.2089
\\
\hline
16&      7.7945E-03 & 2.0184        &1.3486E-02 &2.0326  &5.7844E-02	        &1.1309
\\
\hline
32&		 1.9530E-03& 1.9968      &3.3590E-03&2.0054   &2.7953E-02         &1.0492
\\
\hline
\end{tabular}
\end{center}
\end{table}

\begin{table}[H]
\begin{center}
\caption{Numerical rates of convergence for the $C^{-1}-P_2(T)/P_2(\partial T)/P_1(T)$ element  with the exact solution $\lambda=\sin(\pi x)\cos(\pi y)$ on the unit square domain $\Omega_1$; uniform triangular partitions; convection $\bbeta=[-y, x]$; reaction $c=x+y$; and the parameters $(\tau_1, \tau_2)=(0, 1)$.}\label{NEE:TRI:Case-P2-0-1}
\begin{tabular}{|c|c|c|c|c|c|c|}
\hline
$1/h$        & $\|\epsilon_0\| $ &  order&  $\|\epsilon_b \| $ &  order&  $\|e_h \|$ &order
\\
\hline
1&        8.4480E-01  &&	      1.3885E-00 &                  &2.6279E-01 	&
\\
\hline
2&        1.2875E-01 &	2.7140 &	2.1370E-01 &2.6999       &1.3221E-01	& 0.9910
\\
\hline
4&         1.5824E-02 &3.0244	   &2.5988E-02&3.0396  &	4.6963E-02  &1.4933
\\
\hline
8&         1.4710E-03  &3.4272        &2.4380E-03& 3.4141    &1.8782E-02   &1.3221
\\
\hline
16&     1.5733E-04 & 3.2249     &2.5893E-04 &3.2351  &6.3135E-03	 &1.5729
\\
\hline
32&		1.8166E-05& 3.1145     &2.9573E-05& 3.1302   &2.0863E-03&1.5975
\\
\hline
\end{tabular}
\end{center}
\end{table}

\begin{table}[H]
\begin{center}
\caption{Numerical rates of convergence for the $C^{-1}-P_2(T)/P_2(\partial T)/P_1(T)$ element  with the exact solution $\lambda=\sin(\pi x)\cos(\pi y)$ on the unit square domain $\Omega_1$; uniform triangular partitions; convection $\bbeta=[-y, x]$; reaction $c=x+y$; and the parameters $(\tau_1, \tau_2)=(1, 1)$.}\label{NEE:TRI:Case-P2-1-1}
\begin{tabular}{|c|c|c|c|c|c|c|}
\hline
$1/h$        & $\|\epsilon_0\| $ &  order&  $\|\epsilon_b \| $ &  order&  $\|e_h \|$ &order
\\
\hline
1&     3.8837E-00  &&	        5.4854E-00 &                  &5.6478E-01 	&
\\
\hline
2&     1.7181E-01 &	4.4985 &	2.7930E-01 &4.2957       &1.5540E-01 	      &1.8617
\\
\hline
4&       2.4360E-02 &2.8183	   &3.7158E-02&	2.9101    &5.5710E-02        &1.4800
\\
\hline
8&        1.7514E-03&3.7979     &2.8035E-03&3.7284   &1.8484E-02          &1.5916
\\
\hline
16&       1.7571E-04 & 3.3173       &2.8321E-04 &3.3073  &6.2443E-03	        &1.5657
\\
\hline
32&		 1.9994E-05& 3.1355    &3.1905E-05&3.1500  &2.0630E-02         &1.5978
\\
\hline
\end{tabular}
\end{center}
\end{table}

\begin{table}[H]
\begin{center}
\caption{Numerical rates of convergence for the $C^{-1}-P_2(T)/P_2(\partial T)/P_1(T)$ element  with the exact solution $\lambda=\sin(\pi x)\cos(\pi y)$ on the unit square domain $\Omega_1$; uniform triangular partitions; convection $\bbeta=[-y, x]$; reaction $c=x+y$; and the parameters $(\tau_1, \tau_2)=(0, 0)$.}\label{NEE:TRI:Case-P2-0-0}
\begin{tabular}{|c|c|c|c|c|c|c|}
\hline
$1/h$        & $\|\epsilon_0\| $ &  order&  $\|\epsilon_b \| $ &  order&  $\|e_h \|$ &order
\\
\hline
1&     6.9711E-01  &&	     1.4242E-00 &                  &5.8715E-00 	&
\\
\hline
2&     1.7579E-01 & 1.9875  &3.0395E-01 &2.2282        &1.8269E-00 	      &1.6844
\\
\hline
4&       1.6556E-02 &3.4085	   &2.7948E-02&	3.4430    &4.7677E-01        &1.9380
\\
\hline
8&       1.4546E-03&3.5086    &2.4227E-03&3.5280  &1.2208E-01          &1.9654
\\
\hline
16&      1.5136E-04 & 3.2646       &2.4989E-04 &3.2773  &3.1089E-02	        &1.9734
\\
\hline
32&		1.7656E-05& 3.0998  &2.8785E-05&3.1179  &7.9238E-03         &1.9722
\\
\hline
\end{tabular}
\end{center}
\end{table}

\begin{table}[H]
\begin{center}
\caption{Numerical rates of convergence for the $C^{-1}-P_2(T)/P_2(\partial T)/P_1(T)$ element  with the exact solution $\lambda=\sin(\pi x)\cos(\pi y)$ on the unit square domain $\Omega_1$; uniform triangular partitions; convection $\bbeta=[-y, x]$; reaction $c=x+y$; and the parameters $(\tau_1, \tau_2)=(1, 0)$.}\label{NEE:TRI:Case-P2-1-0}
\begin{tabular}{|c|c|c|c|c|c|c|}
\hline
$1/h$        & $\|\epsilon_0\| $ &  order&  $\|\epsilon_b \| $ &  order&  $\|e_h \|$ &order
\\
\hline
1&       1.0306E-00  &&	      1.8490E-00 &                  &8.8351E-00 	&
\\
\hline
2&       1.6984E-01 &2.6012   &	2.9961E-01 &2.6256      &1.9752E-00 	      &2.1613
\\
\hline
4&       1.8233E-02 &3.2195 	   &3.0485E-02&	3.2969    &4.8981E-01       &2.0117
\\
\hline
8&       1.6511E-03&3.4651     &2.6914E-03&3.5017   &1.2298E-01          &1.9938
\\
\hline
16&      1.6687E-04 & 3.3067        &2.7016E-04 &3.3165  &3.1170E-02	        &1.9801
\\
\hline
32&		1.9382E-05& 3.1059     &3.0957E-05&3.1255  &7.9248E-03         &1.9757
\\
\hline
\end{tabular}
\end{center}
\end{table}

\subsection{Discontinuous convection $\bbeta$}
The numerical tests are conducted for the $C^{-1}-P_1(T)/P_1(\partial T)/P_0(T)$ element on uniform triangular partitions of the unit square domain $\Omega_1$. The exact solution is given by $\lambda=\sin(x)\cos(y)$. The convection $\bbeta(x, y)$ is piece-wise defined in the sense that $\bbeta(x, y)=[1, -1]$ for $y<1-x$ and $\bbeta(x, y)=[-2, 2]$ otherwise. The reaction is $c=1$. The parameters are $(\tau_1, \tau_2)=(1, 1)$, $(\tau_1, \tau_2)=(0, 1)$ and $(\tau_1, \tau_2)=(0, 0)$, respectively. The numerical results in Tables \ref{NE:TRI:£ºL-Case3-1-1-1}-\ref{NE:TRI:£ºL-Case3-1-0-0} show that the convergence rates for $\epsilon_0$ and $\epsilon_b$ in the discrete $L^2$ norm are of the optimal order ${\cal O} (h^2)$, which are consistent with the theory.

\begin{table}[H]
\begin{center}
\caption{Numerical rates of convergence for the $C^{-1}-P_1(T)/P_1(\partial T)/P_0(T)$ element with the exact solution $\lambda=\sin(x)\cos(y)$ on the unit square domain $\Omega_1=(0,1)^2$; uniform triangular partitions; convection $\bbeta=[1, -1]$ for $y<1-x$ and $\bbeta=[-2, 2]$ otherwise; reaction $c=1$; and the parameters $(\tau_1, \tau_2)=(1, 1)$.}\label{NE:TRI:£ºL-Case3-1-1-1}
\begin{tabular}{|c|c|c|c|c|c|c|}
\hline
$1/h$        & $\|\epsilon_0\| $ &  order&  $\|\epsilon_b \| $ &  order&  $\|e_h \|$ &order
\\
\hline
1&       5.2713E-02  &&	      9.8682E-02 &                  &3.0619E-03 	&
\\
\hline
2&       1.1336E-02 &2.2173  &2.0662E-02 &2.2558           &1.5735E-03 	      &0.9604
\\
\hline
4&        2.7930E-03 &2.0210	&4.8760E-03 &2.0832          &	9.8137E-04        &0.6812
\\
\hline
8&       6.8883E-04&2.0196   &  1.1677E-03&2.0620            & 5.1897E-04          &0.9192
\\
\hline
16&      1.7169E-04& 2.0043     & 2.8597E-04&2.0297           &2.6273E-04        &0.9820
\\
\hline
32&		4.2912E-05&2.0004   & 7.0790E-05 & 2.0143       & 1.3186E-04         &0.9946
\\
\hline
\end{tabular}
\end{center}
\end{table}

\begin{table}[H]
\begin{center}
\caption{Numerical rates of convergence for the $C^{-1}-P_1(T)/P_1(\partial T)/P_0(T)$ element with exact solution $\lambda=\sin(x)\cos(y)$ on the unit square domain $\Omega_1$; uniform triangular partitions; convection $\bbeta=[1, -1]$ for $y<1-x$ and $\bbeta=[-2, 2]$ otherwise; reaction $c=1$; and the parameters $(\tau_1, \tau_2)=(0, 1)$.}\label{NE:TRI:£ºL-Case3-1-0-1}
\begin{tabular}{|c|c|c|c|c|c|c|}
\hline
$1/h$        & $\|\epsilon_0\| $ &  order&  $\|\epsilon_b \| $ &  order&  $\|e_h \|$ &order
\\
\hline
1&       3.2120E-02  &&	      7.0574E-02 &                  &9.0720E-03 	&
\\
\hline
2&       1.1795E-02 &1.4452   &2.1695E-02 &1.7018            &2.3864E-03 	      &1.9266
\\
\hline
4&        2.8507E-03 &2.0489	&4.9996E-03 &2.1175           &	9.1523E-04        &1.3826
\\
\hline
8&        7.0230E-04&2.0197   &  1.1936E-03&2.0664             & 4.2927E-04          &1.0922
\\
\hline
16&      1.7520E-04& 2.0045      & 2.9198E-04&2.0315           &2.1125E-04        &1.0229
\\
\hline
32&4.3771E-05		&  2.0009 &7.2224E-05  & 2.0153       & 1.0519E-04         &1.0059
\\
\hline
\end{tabular}
\end{center}
\end{table}

\begin{table}[H]
\begin{center}
\caption{Numerical rates of convergence for the $C^{-1}-P_1(T)/P_1(\partial T)/P_0(T)$ element with the exact solution $\lambda=\sin(x)\cos(y)$ on the unit square domain $\Omega_1$; uniform triangular partitions; convection $\bbeta=[1, -1]$ for $y<1-x$ and $\bbeta=[-2, 2]$ otherwise; reaction $c=1$; and the parameters $(\tau_1, \tau_2)=(0, 0)$.}\label{NE:TRI:£ºL-Case3-1-0-0}
\begin{tabular}{|c|c|c|c|c|c|c|}
\hline
$1/h$        & $\|\epsilon_0\| $ &  order&  $\|\epsilon_b \| $ &  order&  $\|e_h \|$ &order
\\
\hline
1&       4.0936E-02  &&	      9.0019E-02 &                  &1.3976E-02 	&
\\
\hline
2&        1.2259E-02 &1.7396    &2.2582E-02 &1.9951            &3.0632E-03 	      &2.1899
\\
\hline
4&         2.9157E-03 &2.0719	&5.1131E-03 &2.1429           &	1.2530E-03        &1.2896
\\
\hline
8&        7.1768E-04&2.0224    &  1.2183E-03&2.0693             & 5.9790E-04          &1.0675
\\
\hline
16&      1.7876E-04& 2.0053      & 2.9787E-04&2.0321           &2.9519E-04        &1.0183
\\
\hline
32&		 4.4655E-05  & 2.0012     & 7.3676E-05 &2.0154	          & 1.4707E-04           &1.0051
\\
\hline
\end{tabular}
\end{center}
\end{table}

Figures \ref{case1-1}-\ref{case1-2} illustrate the plots of the numerical solution $\lambda_0$ arising from the PD-WG scheme \eqref{al-general}-\eqref{al-general-2} when the $C^{-1}-P_1(T)/P_1(\partial T)/P_0(T)$ element and the $C^{-1}-P_2(T)/P_2(\partial T)/P_1(T)$ element are employed, respectively.
Uniform triangular partitions are employed in the numerical experiments on the unit square domain $\Omega_1$. The configuration of the test problem is as follows: (1) the convection is piece-wisely defined in the sense that $\bbeta(x, y)=[1, -1]$ for $y<1-x$ and $\bbeta(x, y)=[-2, 2]$ elsewhere; (2) the reaction is $c=0$; (3) the load function is $f=0$; and (4) the inflow boundary data is given by $g=1$ on the inflow boundary edge $\{0\}*(0,1)$ and by $g=-1$ on the inflow boundary edge $\{1\}*(0, 1)$. The parameters are given by $(\tau_1, \tau_2)=(1, 1)$. The left ones in Figures \ref{case1-1}-\ref{case1-2} present the contour plots of the numerical solution $\lambda_0$; and the right ones demonstrate the surface plots for the numerical solution $\lambda_0$. It is easy to see that the numerical solution $\lambda_0$ arising from the PD-WG scheme \eqref{al-general}-\eqref{al-general-2} is consistent with the exact solution $\lambda$ of the model problem \eqref{model} in this test problem.

\begin{figure}[h]
\centering
\begin{tabular}{cc}
\resizebox{2.4in}{2.1in}{\includegraphics{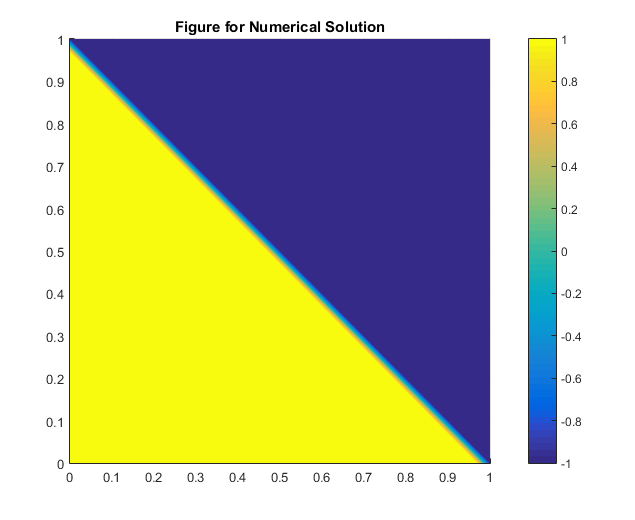}}
\resizebox{2.4in}{2.1in}{\includegraphics{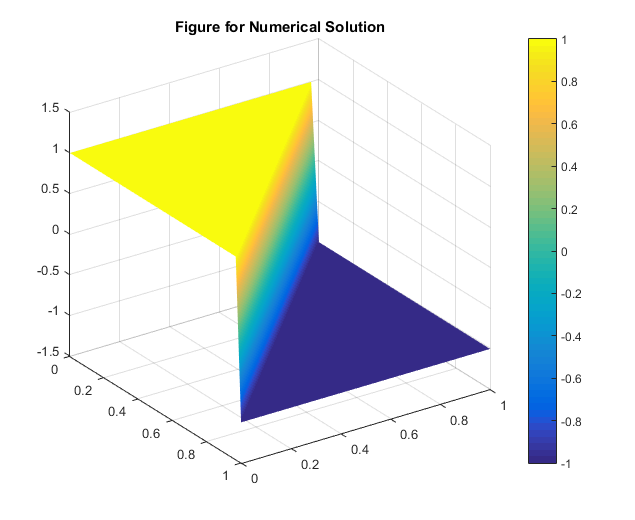}}
\end{tabular}
\caption{Plots of numerical solution $\lambda_0$ on the unit square domain $\Omega_1$; $C^{-1}-P_1(T)/P_1(\partial T)/P_0(T)$ element; uniform triangular partitions; convection $\bbeta=[1, -1]$ for $y<1-x$ and $\bbeta=[-2, 2]$ elsewhere; reaction $c=0$; the load function $f=0$; the inflow boundary data $g=1$ on the inflow boundary edge  $\{0\}*(0,1)$ and $g=-1$ on the inflow boundary edge $\{1\}*(0, 1)$; and $(\tau_1, \tau_2)=(1,1)$. Contour plot (left); surface plot (right).}
\label{case1-1}
\end{figure}

\begin{figure}[h]
\centering
\begin{tabular}{cc}
\resizebox{2.4in}{2.1in}{\includegraphics{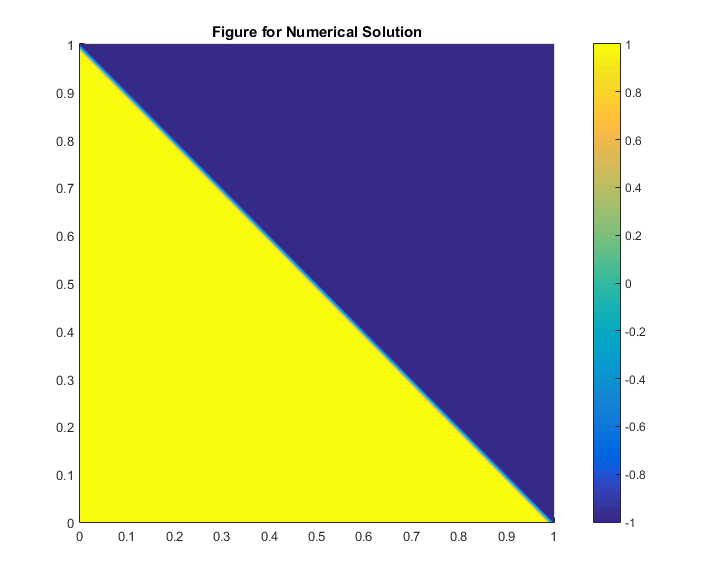}}
\resizebox{2.4in}{2.1in}{\includegraphics{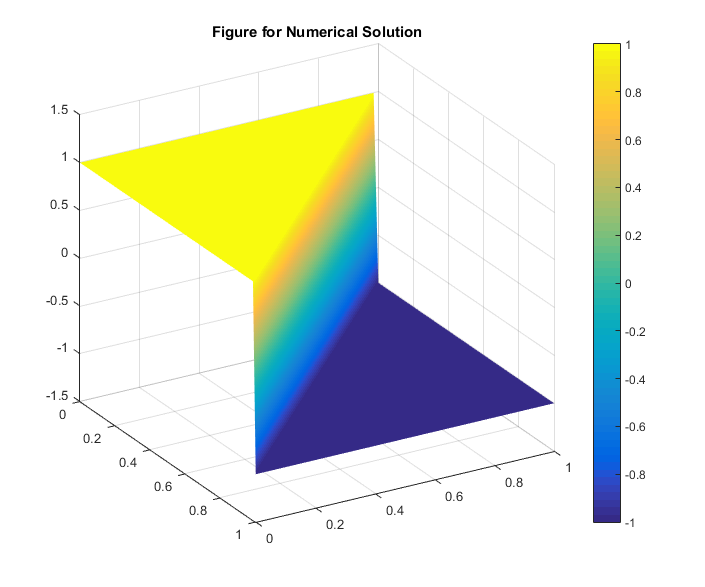}}
\end{tabular}
\caption{Plots of numerical solution $\lambda_0$ on the unit square domain $\Omega_1$; $C^{-1}-P_2(T)/P_2(\partial T)/P_1(T)$ element; uniform triangular partitions; convection $\bbeta=[1, -1]$ for $y<1-x$ and $\bbeta=[-2, 2]$ elsewhere; reaction $c=0$;  the load function $f=0$; the inflow boundary data $g=1$ on the inflow boundary edge  $\{0\}*(0,1)$ and $g=-1$ on the inflow boundary edge $\{1\}*(0, 1)$; and $(\tau_1, \tau_2)=(1,1)$. Contour plot (left); surface plot (right).}
\label{case1-2}
\end{figure}

\subsection{Plots for numerical solution $\lambda_0$ when the exact solution $\lambda$ is unknown}
Figures \ref{case5-1}-\ref{case5-2} show the contour plots of the numerical solution $\lambda_0$ on the uniform triangular partition of the unit square domain $\Omega_1$ when the $C^{-1}-P_1(T)/P_1(\partial T)/P_0(T)$ element and the $C^{-1}-P_2(T)/P_2(\partial T)/P_1(T)$ element are employed respectively. The configuration of this test problem is as follows: the convection vector $\bbeta=[0.5-y, x-0.5]$, the reaction $c=1$, the inflow boundary data $g=\cos(5y)$, and the parameters $(\tau_1, \tau_2)=(1,1)$. The left ones and the right ones in Figures \ref{case5-1}-\ref{case5-2}  demonstrate the contour plots of the numerical solution $\lambda_0$ corresponding to the load function $f=1$ and the load function $f=0$, respectively.

\begin{figure}[h]
\centering
\begin{tabular}{cc}
\resizebox{2.4in}{2.1in}{\includegraphics{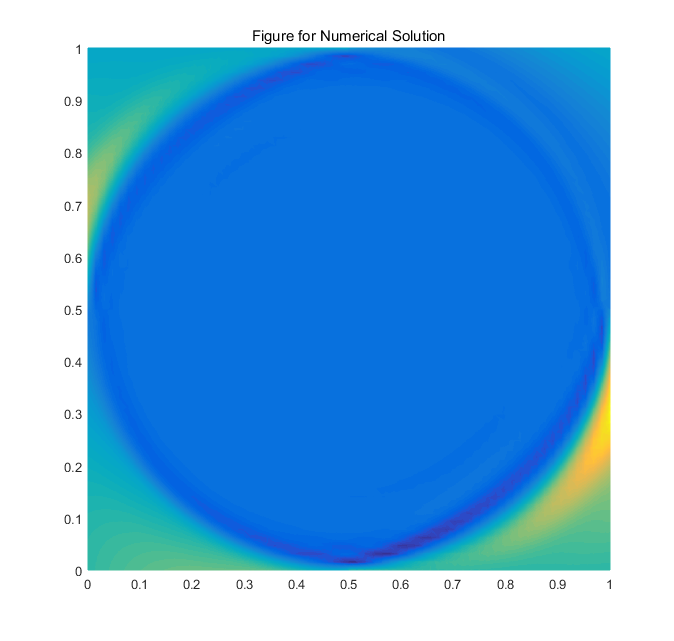}}
\resizebox{2.4in}{2.1in}{\includegraphics{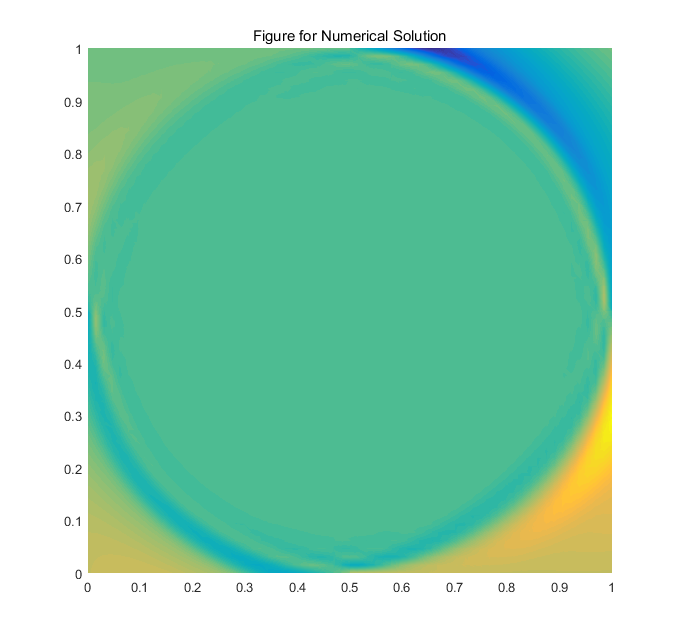}}
\end{tabular}
\caption{Contour plots of numerical solution $\lambda_0$ on the unit square domain $\Omega_1$; $C^{-1}-P_1(T)/P_1(\partial T)/P_0(T)$ element; uniform triangular partitions; the convection $\bbeta=[0.5-y, x-0.5]$; the reaction $c=1$; the inflow boundary data $g=\cos(5y)$; and $(\tau_1, \tau_2)=(1,1)$. The load function $f=1$ (left); the load function $f=0$ (right).}
\label{case5-1}
\end{figure}

\begin{figure}[h]
\centering
\begin{tabular}{cc}
\resizebox{2.4in}{2.1in}{\includegraphics{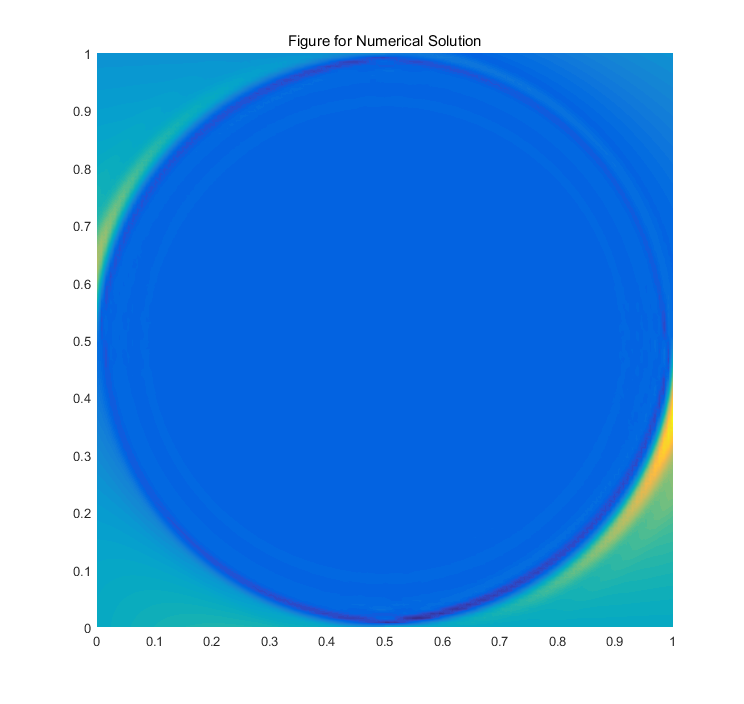}}
\resizebox{2.4in}{2.1in}{\includegraphics{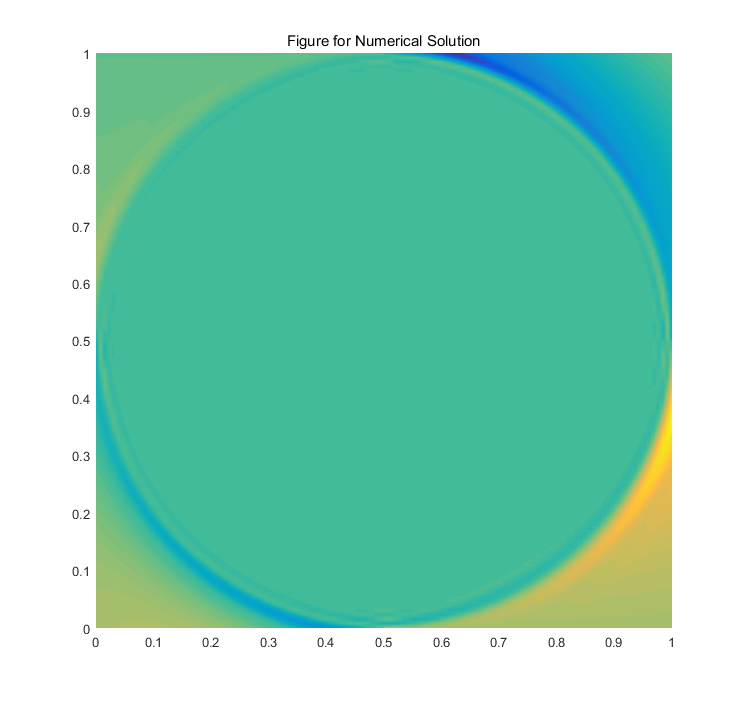}}
\end{tabular}
\caption{Contour plots of numerical solution $\lambda_0$ on the unit square domain $\Omega_1$; $C^{-1}-P_2(T)/P_2(\partial T)/P_1(T)$ element; uniform triangular partitions; convection $\bbeta=[0.5-y, x-0.5]$; reaction $c=1$; the inflow boundary data $g=\cos(5y)$; and $(\tau_1, \tau_2)=(1,1)$. The load function $f=1$ (left); the load function $f=0$ (right).}
\label{case5-2}
\end{figure}

Figures \ref{case3-1}-\ref{case3-2} show the contour plots of the numerical solution $\lambda_0$ on the uniform triangular partition of the unit square domain $\Omega_1$ for the $C^{-1}-P_1(T)/P_1(\partial T)/P_0(T)$ element and the $C^{-1}-P_2(T)/P_2(\partial T)/P_1(T)$ element respectively. The convection vector is piece-wisely defined in the sense that $\bbeta(x, y)=[-y, x]$ for $y<1-x$ and $\bbeta(x, y)=[1-y, x-1]$ otherwise. The  reaction is $c=0$. The inflow boundary data is given by $g=\sin(3x)\cos(5y)$. The parameters are $(\tau_1, \tau_2)=(1,1)$. Figures \ref{case3-1}-\ref{case3-2} demonstrate the contour plots of the numerical solution $\lambda_0$ for the load function $f=1$ (left figures) and the load function $f=0$ (right figures), respectively.

\begin{figure}[h]
\centering
\begin{tabular}{cc}
\resizebox{2.4in}{2.1in}{\includegraphics{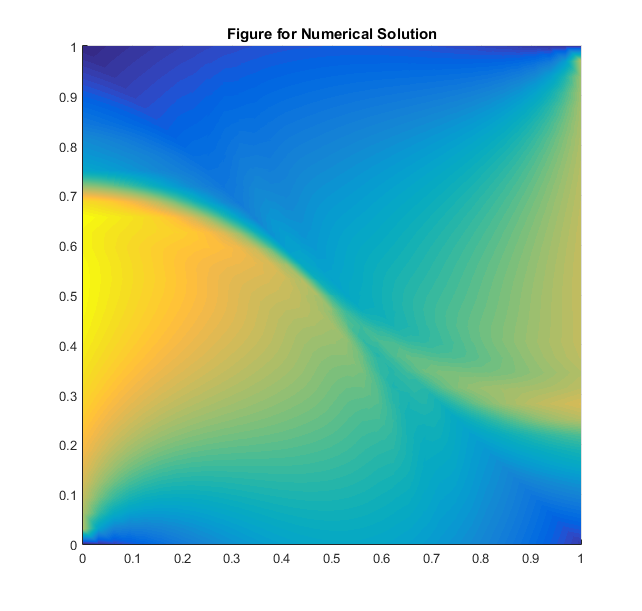}}
\resizebox{2.4in}{2.1in}{\includegraphics{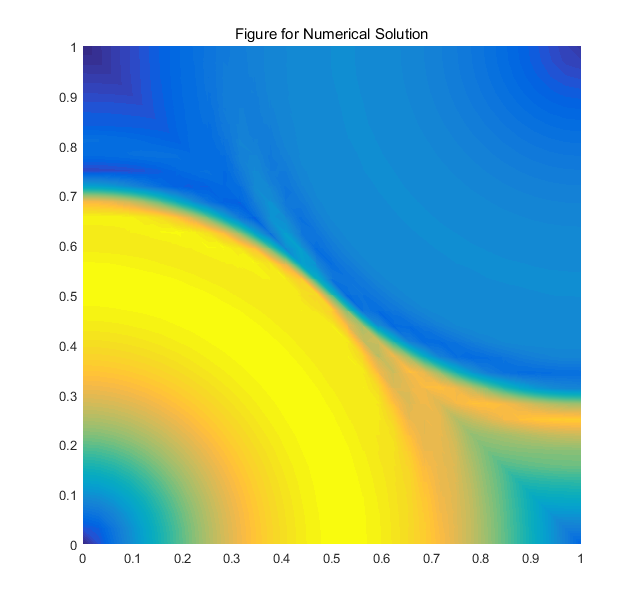}}
\end{tabular}
\caption{Contour plots of numerical solution $\lambda_0$ on the unit square domain $\Omega_1$; $C^{-1}-P_1(T)/P_1(\partial T)/P_0(T)$ element; uniform triangular partitions; convection $\bbeta=[-y, x]$ for $y<1-x$ and $\bbeta=[1-y, x-1]$ otherwise; reaction $c=0$; the inflow boundary data $g=\sin(3x)\cos(5y)$; and $(\tau_1, \tau_2)=(1,1)$. The load function $f=1$ (left); the load function $f=0$ (right).}
\label{case3-1}
\end{figure}

\begin{figure}[h]
\centering
\begin{tabular}{cc}
\resizebox{2.4in}{2.1in}{\includegraphics{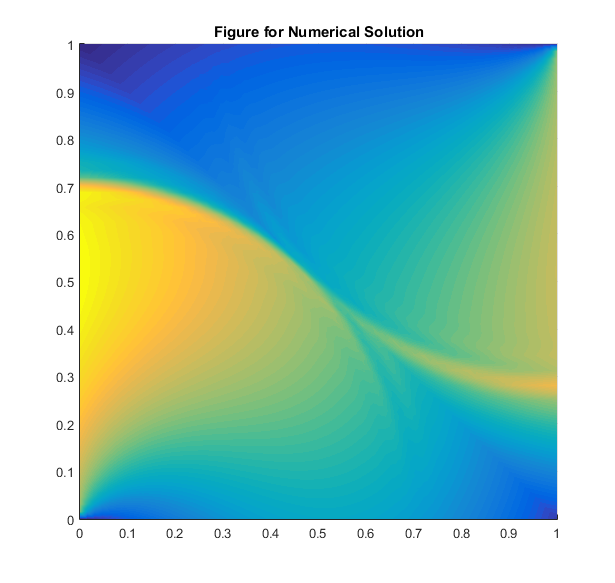}}
\resizebox{2.4in}{2.1in}{\includegraphics{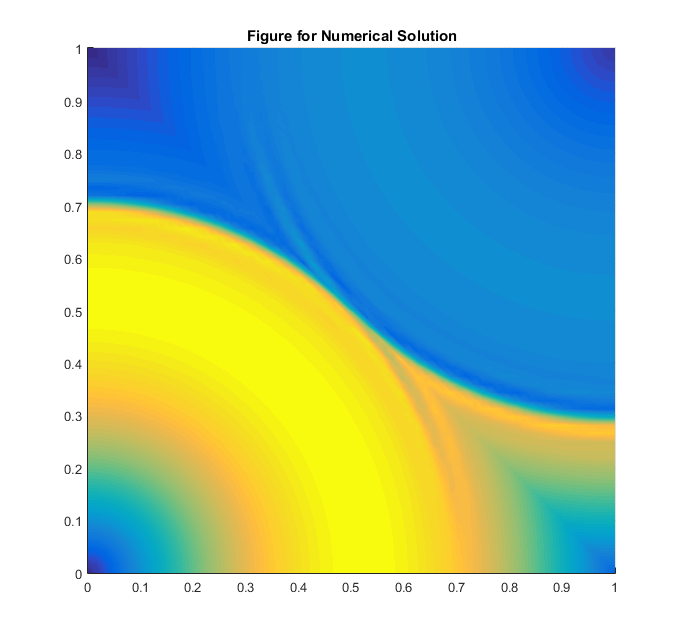}}
\end{tabular}
\caption{Contour plots of numerical solution $\lambda_0$ on the unit square domain $\Omega_1$; $C^{-1}-P_2(T)/P_2(\partial T)/P_1(T)$ element; uniform triangular partitions; convection $\bbeta=[-y, x]$ for $y<1-x$ and $\bbeta=[1-y, x-1]$ otherwise; reaction $c=0$; the inflow boundary data $g=\sin(3x)\cos(5y)$; and $(\tau_1, \tau_2)=(1,1)$. The load function $f=1$ (left); the load function $f=0$ (right).}
\label{case3-2}
\end{figure}

Figures \ref{case4-1}-\ref{case4-2} show the contour plots of the numerical solution $\lambda_0$ on the uniform triangular partition of the L-shaped domain $\Omega_2$. The convection vector is piece-wisely given by $\bbeta(x, y)=[-1, 1]$ for $y<0.5-x$ and $\bbeta(x, y)=[1, -1]$ elsewhere. The reaction is $c=1$. The inflow boundary data is $g=\sin(\pi x)\cos(\pi y)$. The parameters are set as $(\tau_1, \tau_2)=(1,1)$. Figure  \ref{case4-1} demonstrates the contour plots of the numerical solution $\lambda_0$ for the load functions $f=1$ and $f=0$ respectively when the $C^{-1}-P_1(T)/P_1(\partial T)/P_0(T)$ element  is employed. Figure \ref{case4-2} illustrates the contour plots of the numerical solution $\lambda_0$ for the load functions $f=1$ and $f=0$ respectively when the $C^{-1}-P_2(T)/P_2(\partial T)/P_1(T)$ element is used.

\begin{figure}[h]
\centering
\begin{tabular}{cc}
\resizebox{2.4in}{2.1in}{\includegraphics{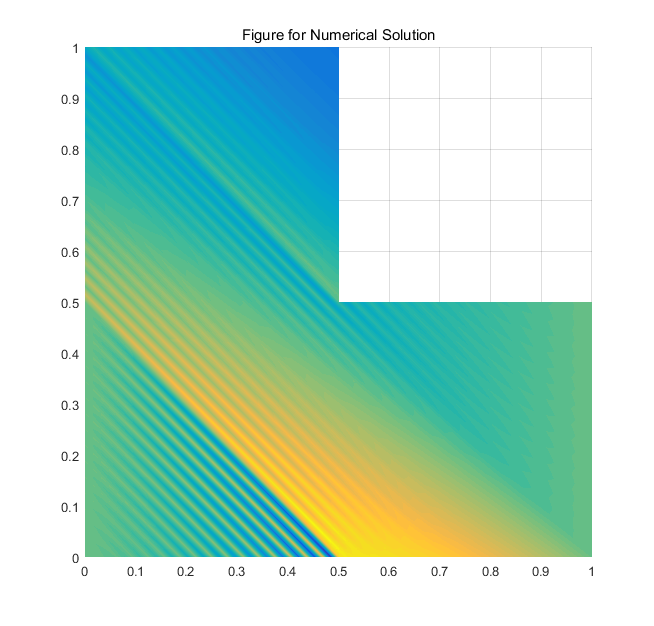}}
\resizebox{2.4in}{2.1in}{\includegraphics{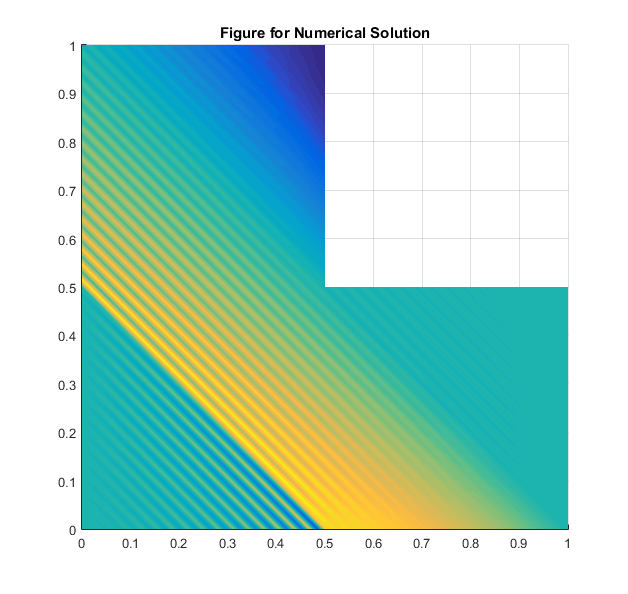}}
\end{tabular}
\caption{Contour plots of numerical solution $\lambda_0$ on the L-shaped domain $\Omega_2$; $C^{-1}-P_1(T)/P_1(\partial T)/P_0(T)$ element; uniform triangular partitions; convection vector $\bbeta=[-1, 1]$ for $y<0.5-x$ and $\bbeta=[1, -1]$ elsewhere; reaction $c=1$; the inflow boundary data $g=\sin(\pi x)\cos(\pi y)$; and $(\tau_1, \tau_2)=(1,1)$. The load function $f=1$ (left); the load function $f=0$ (right).}
\label{case4-1}
\end{figure}

\begin{figure}[h]
\centering
\begin{tabular}{cc}
\resizebox{2.4in}{2.1in}{\includegraphics{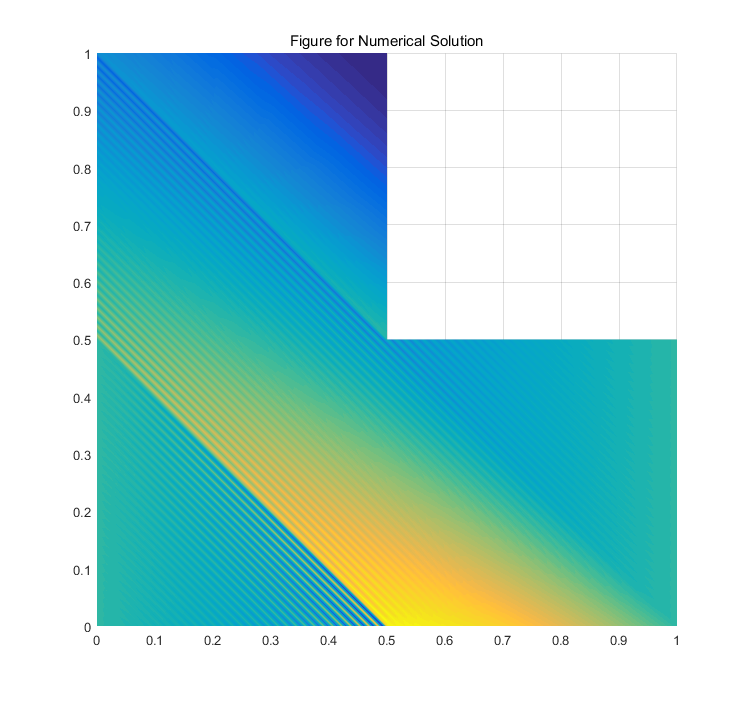}}
\resizebox{2.4in}{2.1in}{\includegraphics{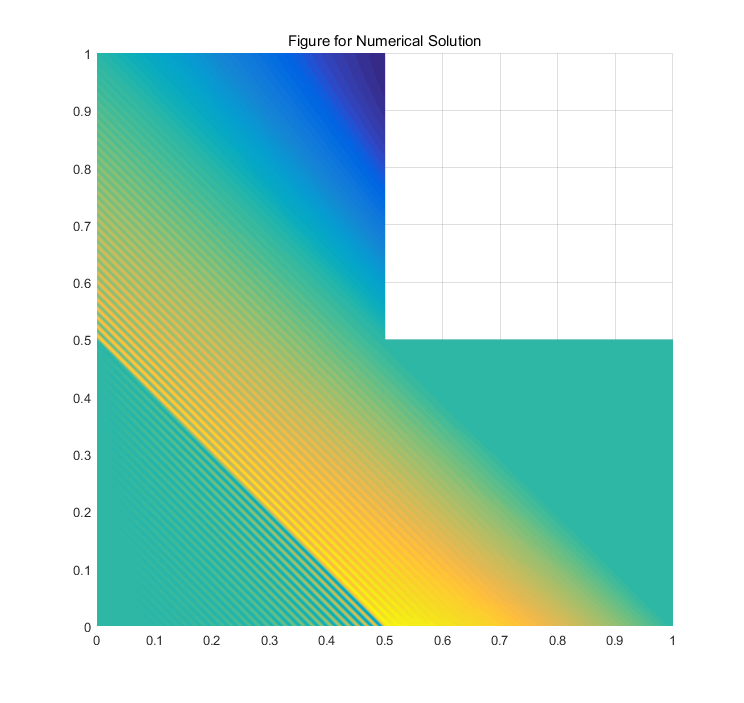}}
\end{tabular}
\caption{Contour plots of numerical solution $\lambda_0$ on the L-shaped domain $\Omega_2$; $C^{-1}-P_2(T)/P_2(\partial T)/P_1(T)$ element; uniform triangular partitions; convection $\bbeta=[-1, 1]$ for $y<0.5-x$ and $\bbeta=[1, -1]$ elsewhere; reaction $c=1$; the inflow boundary data $g=\sin(\pi x)\cos(\pi y)$; and $(\tau_1, \tau_2)=(1,1)$. The load function $f=1$ (left); the load function $f=0$ (right).}
\label{case4-2}
\end{figure}

Figures \ref{case2-1}-\ref{case2-2} show the contour plots of the numerical solution $\lambda_0$ on the uniform triangular partition of the cracked unit square domain $\Omega_3$. The convection vector is given by $\bbeta=[0.5-y, x-0.5]$; the reaction is $c=x-y$; the inflow boundary data is $g=\cos(5y)$; and the parameters are $(\tau_1, \tau_2)=(1,1)$. Figures  \ref{case2-1}-\ref{case2-2} demonstrate the contour plots of the numerical solution $\lambda_0$ for the load function $f=1$ (left ones) and the load function $f=0$ (right ones) when the $C^{-1}-P_1(T)/P_1(\partial T)/P_0(T)$ element and  the $C^{-1}-P_2(T)/P_2(\partial T)/P_1(T)$ element are employed respectively.

\begin{figure}[h]
\centering
\begin{tabular}{cc}
\resizebox{2.4in}{2.1in}{\includegraphics{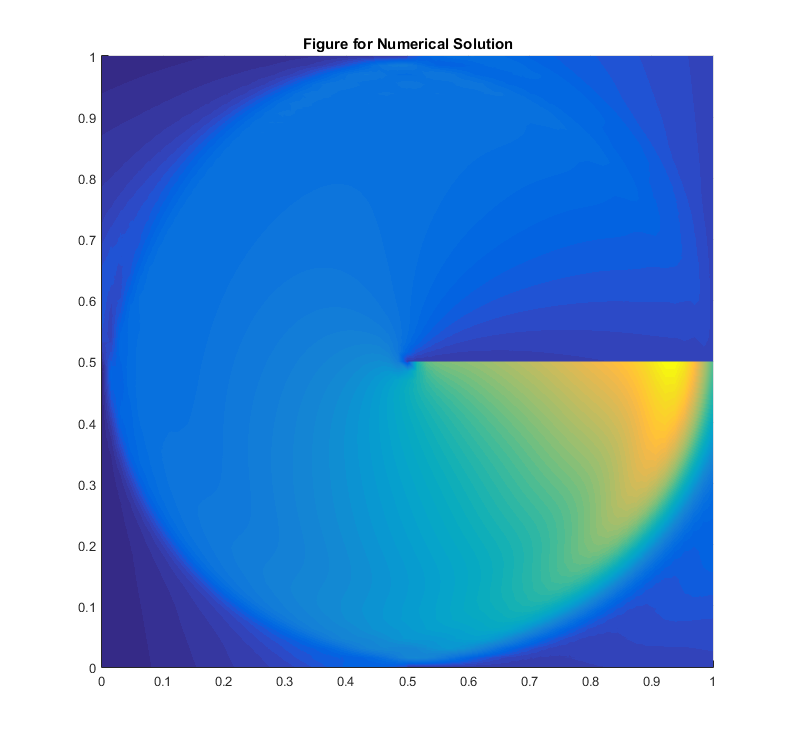}}
\resizebox{2.4in}{2.1in}{\includegraphics{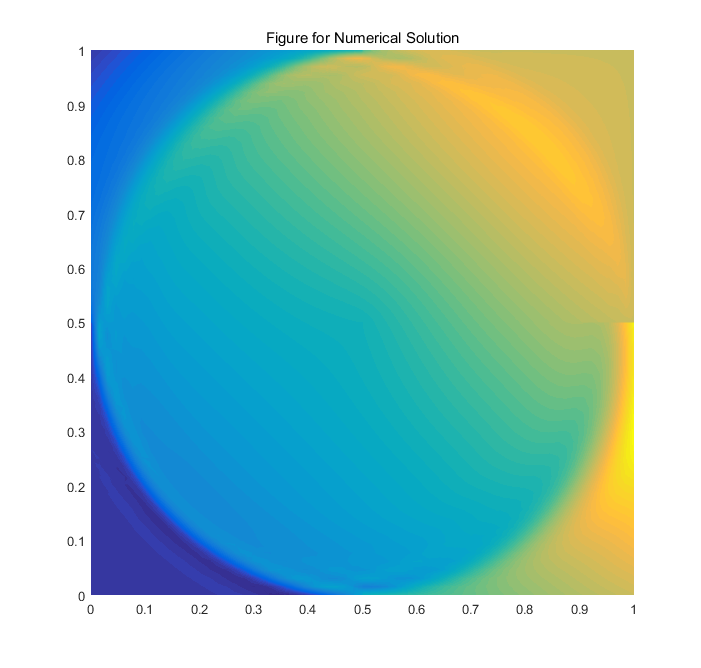}}
\end{tabular}
\caption{Contour plots of numerical solution $\lambda_0$ on the cracked square domain $\Omega_3$; $C^{-1}-P_1(T)/P_1(\partial T)/P_0(T)$ element; uniform triangular partitions; convection $\bbeta=[0.5-y, x-0.5]$; reaction $c=x-y$; the inflow boundary data $g=\sin(x)$; and $(\tau_1, \tau_2)=(1,1)$. The load function $f=1$ (left); the load function $f=0$ (right).}
\label{case2-1}
\end{figure}

\begin{figure}[h]
\centering
\begin{tabular}{cc}
\resizebox{2.4in}{2.1in}{\includegraphics{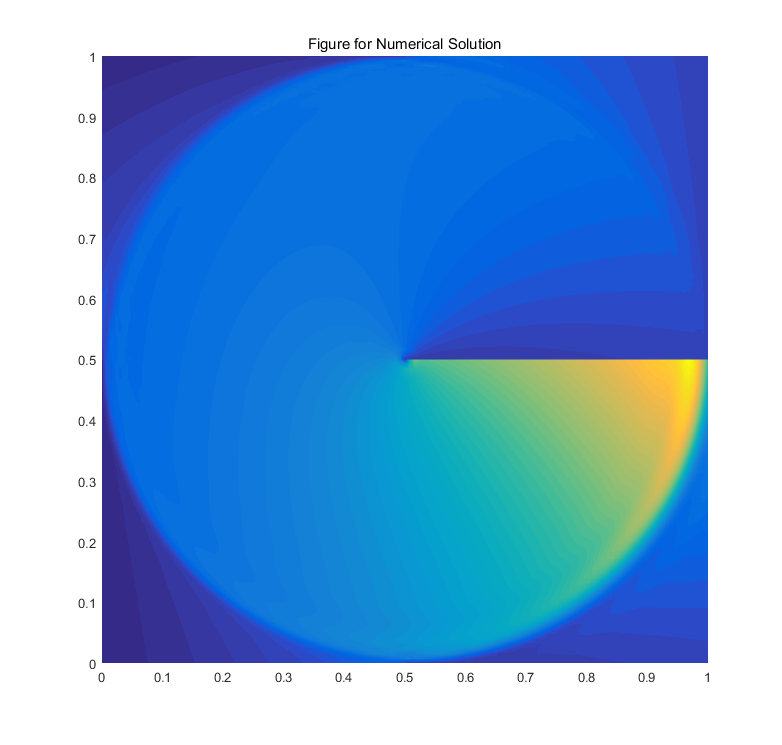}}
\resizebox{2.4in}{2.1in}{\includegraphics{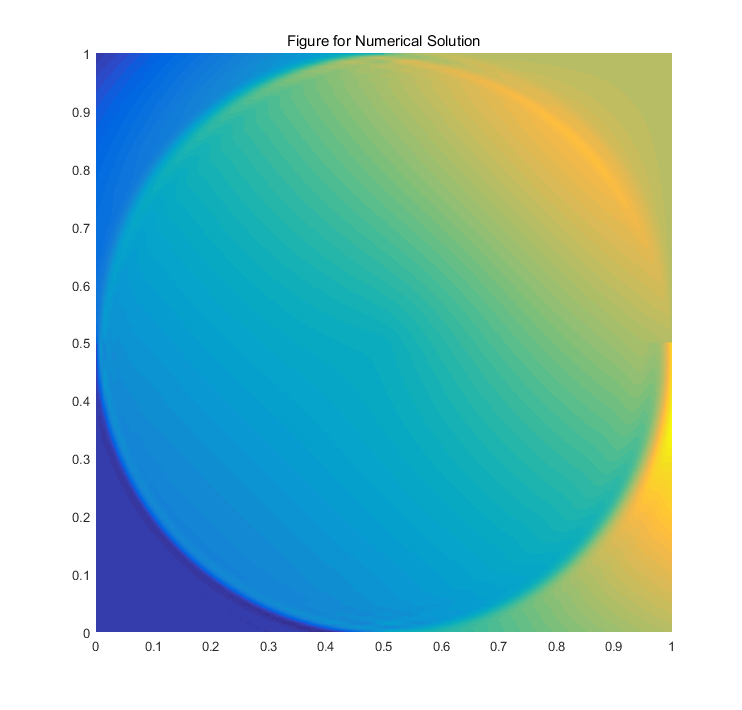}}
\end{tabular}
\caption{Contour plots of numerical solution $\lambda_0$ on the cracked square domain $\Omega_3$; $C^{-1}-P_2(T)/P_2(\partial T)/P_1(T)$ element; uniform triangular partitions; convection $\bbeta=[0.5-y, x-0.5]$; reaction $c=x-y$; the inflow boundary data $g=\sin(x)$; and $(\tau_1, \tau_2)=(1,1)$. The load function $f=1$ (left); the load function $f=0$ (right).}
\label{case2-2}
\end{figure}

In summary, the numerical performance of the primal-dual weak Galerkin scheme (\ref{al-general})-(\ref{al-general-2}) for solving the first-order linear convection problem \eqref{model} is consistent with the theory developed in this paper.  The numerical results reveal optimal-order of convergence for all the test problems. We are confident that the PD-WG scheme is a stable, accurate, and convergent numerical method for the first-order linear convection problem in non-divergence form.

\vfill\eject

\end{document}